\documentclass[11pt]{amsart}
\usepackage{amsmath,amsthm,amsfonts,amscd,flafter,epsf}
\usepackage{amssymb,graphicx,color}
\usepackage{epsfig,manfnt,mathrsfs}
\usepackage{amsfonts}
\usepackage{amssymb}
\usepackage{amsmath}
\usepackage{tikz}
\usepackage{tikz-cd}
\usetikzlibrary{matrix,calc,arrows}

\usepackage{amsthm}
\usepackage{latexsym}
\usepackage{amscd}
\usepackage{flafter}
\usepackage{epsf}
\usepackage{tikz,tikz-cd}
\usepackage[a4paper,footskip=2.2cm,headheight=20pt,top=2.2cm,
bottom=2.2cm,right=2.2cm, left=2.2cm]{geometry}

\input{diagram}
\newcommand{\GV}{\mathfrak{gv}}

\newcommand{\HC}{\check{\mathrm{H}}}

\newcommand{\Adj}{\mathrm{Adj}}

\newcommand{\wt}{\widetilde}

\newcommand{\h}{\mathsf{t}}

\newcommand{\Euler}{\mathsf{e}}
\newcommand{\cg}{\mathsf{c}}
\newcommand{\cq}{{c}}
\newcommand{\bg}{\mathsf{b}}

\newcommand{\Groupq}{{\mathcal{Q}}_{0\ra 0}}
\newcommand{\Groupg}{{\mathcal{G}}_{0\ra 0}}
\newcommand{\Groupoid}{{\mathcal{H}}}
\newcommand{\Groupoidg}{{\mathcal{G}}}
\newcommand{\Groupoidq}{{\mathcal{Q}}}
\newcommand{\Algebroidg}{{\mathfrak{g}}}
\newcommand{\Algebroidq}{{\mathfrak{q}}}
\newcommand{\Obj}{{\mathrm{Obj}}}
\newcommand{\Mor}{{\mathrm{Mor}}}
\newcommand{\sg}{\mathsf{g}}
\newcommand{\sq}{\mathsf{q}}
\newcommand{\simg}{{\sim_\mathsf{g}}}
\newcommand{\simq}{{\sim_\mathsf{q}}}
\newcommand{\Morg}{{\mathrm{Mor}^{\mathsf{g}}}}
\newcommand{\Morq}{{\mathrm{Mor}^{\mathsf{q}}}}

\newcommand{\algebrag}{\mathfrak{g}_{0\ra 0}}
\newcommand{\algebraq}{\mathfrak{q}_{0\ra 0}}
\newcommand{\algebroidg}{{\mathfrak{g}}_{0}}
\newcommand{\algebroidq}{{\mathfrak{q}}_{0}}
\newcommand{\CM}{\mbox{\Large $\wedge$}}
\newcommand{\CMTT}{\mbox{\Large $\wedge$}^{\sharp}}
\newcommand\Moduli{{\Mod(M,C^\infty(\R),\Diff^+(\R))}}

\newcommand\Modulin[1]{{\Mod(M_{#1},C^\infty(\R),\Diff^+(\R))}}

\newcommand\Modulig{{\Mod(M,\algebroidg,\Groupg)}} 
\newcommand\Moduliig{{\Mod(M,\algebrag,\Groupg)}} 
\newcommand\Moduliiq{{\Mod(M,\algebraq,\Groupq)}} 

\newcommand\Moduliign[1]{{\Mod(M_{#1},\algebrag,\Groupg)}}
\newcommand\Moduliiqn[1]{{\Mod(M_{#1},\algebraq,\Groupq)}}
\newcommand\Cordoba{{\CMTT(M,{C^\infty(\R)})}} 
\newcommand\Cordobag{{\CM(M,\algebroidg)}} 
\newcommand\Cordobaq{{\CM(M,\algebroidq)}} 
\newcommand\Cordobaig{{\CM(M,\algebrag)}} 
\newcommand\Cordobaiq{{\CM(M,\algebraq)}} 
\newcommand\CordobaS{{\CM(M,{C^\infty(S^1)})}} 
 
\newcommand\Cordoban[1]{{\CMTT(M_{#1},{C^\infty(\R)})}}

\newcommand{\ConcordiaS}{\mathcal{M}(M,C^\infty(S^1),\Diff^+(S^1))}

\newcommand\Gauge{{\Omega^0(M,{\Diff^+(\R)})}}

\newcommand\Gaugeg{{\Omega^0(M,\Groupg)}}
\newcommand\Gaugeq{{\Omega^0(M,\Groupq)}}
\newcommand{\Diff}{\mathrm{Diff}}

\newcommand{\oA}{{\mathsf{A}}}
\newcommand{\oB}{{\mathsf{B}}}
\newcommand{\oC}{{\mathsf{C}}}
\newcommand{\oD}{{\mathsf{D}}}
\newcommand{\oE}{{\mathsf{E}}}
\newcommand{\oX}{{\mathsf{X}}}
\newcommand{\oY}{{\mathsf{Y}}}
\newcommand{\oZ}{{\mathsf{Z}}}
\newcommand{\oa}{{\mathsf{X}}}
\newcommand{\ob}{{\mathsf{Y}}}
\newcommand{\oc}{{\mathsf{Z}}}
\newcommand{\od}{{\mathsf{W}}}

\newcommand{\qA}{A}
\newcommand{\qY}{Y}
\newcommand{\qa}{X}

\newcommand{\quot}{\mathsf{T}}

\newcommand{\proj}{\mathsf{proj}}

\newcommand{\Hall}{\mathrm{H}}
\newcommand{\HG}{\mathrm{H}_{\Algebroidg}}
\newcommand{\HQ}{\mathrm{H}_{\Algebroidq}}

\newcommand{\ra}{\rightarrow}

\newcommand{\Ibb}{\mathbb{I}}

\newcommand{\Rbb}{\mathbb{R}}

\newcommand{\R}{\Rbb}

\newcommand{\Z}{\mathbb{Z}}

\newcommand{\Bcal}{\mathcal{B}}
\newcommand{\Ccal}{\mathcal{C}}
\newcommand{\Dcal}{\mathcal{D}}

\newcommand{\Fcal}{\mathcal{F}}

\newcommand{\Lcal}{\mathcal{L}}
\newcommand{\Mcal}{\mathcal{M}}

\newcommand{\Ucal}{\mathcal{U}}

\newcommand{\Fscr}{\mathscr{F}}

\newcommand{\Fol}{\Fscr}

\newcommand{\gfrak}{\mathfrak{g}}

\newcommand{\ov}{\overline}

\def\endproof{\relax\ifmmode\expandafter\endproofmath\else
  \unskip\nobreak\hfil\penalty50\hskip.75em\hbox{}\nobreak\hfil\bull
  {\parfillskip=0pt \finalhyphendemerits=0 \bigbreak}\fi}
\def\endproofmath$${\eqno\bull$$\bigbreak}
\def\bull{\vbox{\hrule\hbox{\vrule\kern3pt\vbox{\kern6pt}\kern3pt\vrule}\hrule}}

\newcommand{\Ker}{\mathrm{Ker}}

\newcommand{\Hom}{\mathrm{Hom}}

\newtheorem{thm}{Theorem}[section]
\newtheorem{prop}[thm]{Proposition}
\newtheorem{cor}[thm]{Corollary}

\newtheorem{lem}[thm]{Lemma}

\newtheorem{defn}[thm]{Definition}

\newtheorem{remark}[thm]{Remark}

\newcommand{\Mod}{\Mcal}

\newcommand{\Group}{\mathfrak{G}}

\newcommand{\lra}{\longrightarrow}

\begin{document}
\title[Gauge theory and foliations I]
{Gauge theory and foliations I;\\ germ cords versus quantum cords}%
\author{Mehrzad Ajoodanian, Eaman Eftekhary}%
\address{School of Mathematics, Institute for Research in 
Fundamental Science (IPM),
P. O. Box 19395-5746, Tehran, Iran}%
\email{mehrzad@ipm.ir}
\email{eaman@ipm.ir}
\date{December 2017}%
\begin{abstract}
We apply gauge theory to study the space $\Fcal_k(M)$ of  
smooth codimension-$k$ framed foliations  
on a smooth manifold $M$. 
The quotient of Maurer-Cartan elements
by the action of an infinite dimensional non-abelian gauge groupoid
forms a moduli space, which contains $\Fcal_k(M)$ as a subspace.
The  notion of holonomy is naturally extended to this moduli space 
and the cohomology theory associated with points of this moduli 
space  which correspond to non-singular foliations coincides with 
Bott cohomology. The quotient of the moduli space under 
concordance is identified as the space of homotopy classes of maps 
to the classifying spaces $B\Gamma^\sg_k$ and $B\Gamma^\sq_k$.
While $B\Gamma^\sg$ is a classic and has been studied since 
Haefliger, $B\Gamma^\sq$  (which is a quotient of $B\Gamma^\sg$) 
carries a simpler topology and offers a rival theory. 
\end{abstract}
\maketitle

\section{Introduction}
Foliations grew out of Poincare's qualitative theory of differential 
equations and Ehresmann's connection theory on vector bundles. The 
central idea in both, is the notion of holonomy or monodromy which 
dates even farther back to the time of Cauchy and Riemann.
Gauge theory and foliations crossed paths several times. Perhaps the 
first happened in the 40's as Ehrasmann developed the connection 
theory of vector bundles and generalized Poincare's holonomy of 
a loop lying on a leaf of a foliation. In 70's, the Godbillon-Vey 
invariants of foliations were introduced \cite{Godbillon-Vey}
few years prior to the Chern-Simons functional in gauge theory. The 
similarity between the two intrigued and inspired mathematicians. 
Most notably, R. Bott  introduced the notion of (partial) Bott 
connection on the normal bundle of a foliation, 
c.f.~\cite{Bott-lectures}. Bott connection is flat on the 
leaves, generalizing the Reeb's class (which measures the transverse 
holonomy expansion \cite{Ghys-GV-1}) to higher codimensions. 
Theory of foliations went through rapid developments in 70's,
as Thurston proved his important existence result 
\cite{Thurston-Existence} and Haefliger structures provided a 
framework for classification of foliations up to concordance 
\cite{Haefliger-lectures,Haefliger-1}.
Mather \cite{Mather-on-Haefliger,Mather-homology} and 
Thurston \cite{Thurston-classification} proved important theorems 
about the classifying space of Haefliger structures, which may be 
compared with classification results for homogenous foliations
and flat connections \cite{Characteristic-classes,Blumenthal}. 
These classification results 
linked the study of foliations on spheres to the homotopy theory 
of the classifying space of Haefliger structures and resulted in 
several existence and non-existence theorems 
\cite{Haefliger-analytic,Thurston-GV, Thurston-classification}. 
Despite these similarities, a path
from foliations back to gauge theory has been missing.\\

Contrary to Ehresmann who saw a 
foliation in a flat  connection, we associate flat connections to  
framed foliations. Such connections turn out 
to be gauge equivalent. 
We  examine Frobenius equation from a gauge theoretic point of view 
and identify nonabelian infinite dimensional gauge groups. 
If a smooth codimension-one foliation $\Fol$ on $M$ is given by a 
$1$-form $a_0\in\Omega^0(M,\R)$, Frobenius  equation implies that
$da_0=a_1a_0$ for another $1$-form $a_1\in \Omega^1(M,\R)$ which
is determined up to scalar multiples of $a_0$. The process may 
be repeated to obtain the $1$-forms
$a_2,a_3,a_4,\cdots \in\Omega^1(M,\R)$ so that their derivatives
satisfy a sequence of equations, starting with $da_1=a_2a_0$ 
and $da_2=a_3a_0+a_2a_1$. 
One motivation for the current work is to present the  
{\emph{Godbillon-Vey sequences}}  $(a_0,a_1,a_2,...)$ 
\cite{Godbillon, Ghys-GV-1} as geometric objects (gauge fields) and 
observe their local symmetries through gauge action.  Gauge theory then 
suggests a study of the moduli space of foliations, while we are
 lead to welcome certain singular objects (foliations).
Two groupoids are encountered along the journey, as we consider 
smooth framed foliation of codimension $k$  on a manifold $M$. 
The first groupoid is $\Groupoidq_k$ which consists 
of formal power series of the form 
\[\qY=\sum_{i=1}^k
\sum_I
y_{i,I}(\h_1-x_1)^{i_1}\cdots (\h_k-x_k)^{i_k}\partial_i
=\sum_{i,I}y_{i,I}(\h-x)^I\partial_i\] 
in the formal  
variables $\h_1,\ldots,\h_k$ with $\det(y_{i,j})_{i,j=1}^k>0$. 
The index $I$ runs over the $k$-tuples $(i_1,\ldots,i_k)\in\Z^k$ of 
non-negative integers and $\partial_1,\ldots,\partial_k$ denote 
the standard unit vectors of $\R^k$.
The power series $\qY$ is 
realized as an arrow from the source 
$x=(x_1,\ldots,x_k)\in\R^k=\Obj(\Groupoidq_k)$ to 
the target $y=(y_{1,\emptyset},\ldots,y_{k,\emptyset})
\in\R^k=\Obj(\Groupoidq_k)$. 
The groupoid $\Groupoidq_k$ acts on its Lie 
algebra $\Algebroidq_k$ of all power series of the form 
$\qA=\sum_{i,I} a_{i,I}(\h-x)^I\partial_i$. Alternatively, a 
parallel theory is created if we replace the gauge groupoid
$\Groupoidq_k$ with the groupoid $\Groupoidg_k$ of germs of local 
diffeomorphisms of $\R^k$  and replace $\Algebroidq_k$ with the 
algebroid $\Algebroidg_k$ of germs of smooth $\R^k$-valued maps at 
points of $\R^k$. For every smooth manifold $M$, 
$\Omega^0(M,\Groupoidq_k)$ acts on  flat $\Algebroidq_k$-valued 
$1$-forms, which are called {\emph{quantum cords}} and are denoted 
by $\CM(M,\Algebroidq_k)$. Similarly,  $\Omega^0(M,\Groupoidg_k)$ 
acts on the space $\CM(M,\Algebroidg_k)$ of 
{\emph{germ cords}}. We write $\Gamma^\sg_k$ and $\Gamma^\sq_k$ for 
$\Groupoidg_k$ and $\Groupoidq_k$ if we want to emphasize 
that the discrete topology is chosen on the space of arrows 
from $x\in\R^k$ to $y\in\R^k$. If $(a_0,a_1,a_2,\ldots)$ is the 
Godbillon-Vey sequence associated with a transversely oriented 
codimension one foliation $\Fol$ on $M$, 
$\qA=\sum_n a_n\h^n\in\CM(M,\Algebroidq_1)$ would be a quantum 
cord. Our main theorem  may be stated as follows:

\begin{thm}\label{thm:classification-general}
The space $\Fcal_k(M)$ of smooth framed codimension $k$ foliations 
of $M$  is embedded in the moduli 
space of germs cords and quantum cords (upto gauge action)
while there  are natural classification maps $\cg$ and $\cq$ 
from these moduli spaces to the space of all Haefliger 
$\Gamma^\sg_k$-structures and $\Gamma^\sq_k$-structures
(respectively), which make the following diagram commutative. 
\begin{equation}\label{eq:T-is-surjective-general}
\begin{diagram}
\Fcal_k(M)&\rInto{\ \ \ \ \Ibb^\sg_k\ \ \ \ } &
\Mod(M,\Algebroidg_k,\Groupoidg_k)=\CM(M,\Algebroidg_k)/
\Omega^0(M,\Groupoidg_k)&\rTo{\ \ \cg\ \ }& \HC^1(M,\Gamma^\sg_k)\\
\dTo{Id}&&\dTo{\quot}&&\dTo{\quot}\\
\Fcal_k(M)&\rInto{\ \ \ \ \Ibb^\sq_k\ \ \ \ } &
\Mod(M,\Algebroidq_k,\Groupoidq_k)=\CM(M,\Algebroidq_k)/
\Omega^0(M,\Groupoidq_k)&\rTo{\ \ \cq\ \ }& \HC^1(M,\Gamma^\sq_k)\\
\end{diagram}
\end{equation} 
The maps $\quot$ are obtained by taking Taylor expansions. 
The classification maps $\cg$ and $\cq$ induce the maps
\begin{align*}
\cg:\Ccal(M,\Algebroidg_k)=\CM(M,\Algebroidg_k)/\sim_\sg\ra 
[M,B\Gamma^\sg_k]\ \ \ \text{and}\ \ \ 
\cq:\Ccal(M,\Algebroidq_k)=\CM(M,\Algebroidq_k)/\sim_\sq\ra 
[M,B\Gamma^\sq_k]
\end{align*}
from the germs concordia (concordance classes of germ cords) and 
quantum concordia (concordance classes of quantum cords) to the 
spaces of homotopy classes of maps from $M$ to the classifying 
spaces $B\Gamma^\sg_k$ and $B\Gamma^\sq_k$, respectively.
\end{thm}

Let us fix the codimension $k$ and drop it from the notation for 
simplicity. Let us denote the set of arrows in 
$\Groupoidg=\Groupoidg_k$
which start at source $x\in\R^k$ and end at  target $y\in\R^k$ 
by $\Groupoidg_{x\ra y}$. Similarly, define $\Groupoidq_{x\ra y}$ and 
note that $\Groupg$ and $\Groupq$ are both topological groups,
where $\Groupq$ has the structure of an infinite dimensional 
Lie group. Let $\algebraq$ and $\algebrag$ denote the 
{\emph{Lie algebras}} of  $\Groupq$ and $\Groupg$, respectively. 
The cords in $\Cordobaiq$ and $\Cordobaig$ are called the 
{\emph{impotent}} quantum and germ cords respectively.
The algebras $\algebrag$ and $\algebraq$ are (respectively) 
sub-algebras of the fibers $\algebroidg$ and $\algebroidq$ of 
$\Algebroidg$ and $\Algebroidq$ over $0\in\R^k$.
An important source of examples of impotent quantum and germ cords
is the restriction of a quantum or germ cord to the leaves of 
a foliation $\Fol\in\Fcal_k(M)$.

\begin{thm}\label{thm:classification-impotent}
For every smooth manifold $M$, there are 
natural bijections
\begin{align*}
&\rho_{M}^\sg:\Cordobaig/\Omega^0(M,\Groupg)\lra 
\Hom(\pi_1(M),\Groupg)/\Groupg\ \ \ \ \text{and}\\
&\rho_{M}^\sq:\Cordobaiq/\Omega^0(M,\Groupq)\lra 
\Hom(\pi_1(M),\Groupq)/\Groupq.
\end{align*}
If $L$ is a leaf of a foliation $\Fol\in\Fcal_k(M)$ 
which corresponds to a germ cord $\oA\in\Cordobag$, the 
conjugacy class of the homomorphism 
$\rho_{L}^\sg(\oA|_{L}):\pi_1(L)\ra \Groupg$
gives the holonomy of the leaf $L$. 
\end{thm}   

This theorem gives a way to generalize the notion of leaves and 
their holonomy for singular foliations, i.e. arbitrary gauge 
equivalence classes 
\[\Fol\in\Cordobag/\Omega^0(M,\Groupg)
=\CM(M,\Algebroidg)/\Omega^0(M,\Algebroidg).\]
 A {\emph{leaf-like}} map for 
the singular foliation $\Fol$, which corresponds to some 
$\oA\in\Cordobag$, is a diffeomorphism $f:L\ra M$ from 
a smooth manifold $L$ to $M$ such that $f^*\oA$ is impotent. 
The definition is  independent of the choice of the representative
$\oA$ for the singular foliation $\Fol$. Associated with a leaf-like
map $f:L\ra M$, we obtain the conjugacy class of a holonomy map
\[\rho_{\Fol,L}=\rho_{L}^\sg(\oA|_{L})
\in\Hom(\pi_1(L),\Groupg)/\Groupg\simeq 
\Mod(L,\algebrag,\Groupg).\]
This notion of holonomy generalizes the usual holonomy map for 
the leaves of non-singular foliations. It is nice to compare this approach 
with the approaches of \cite{Holonomy-Groupoid}. \\

Every flat $1$-form may be used to define a twisted differential 
on differential forms. In particular, each 
$\oA\in\CM(M,\Algebroidg_k)$ gives a differential
\[\nabla_{\oA}:\Omega_{s_\oA}^*(M,\Algebroidg)\ra 
\Omega^{*+1}_{s_\oA}(M,\Algebroidg)\]
which satisfies $\nabla_\oA\circ\nabla_\oA=0$. 
Here, the subscript $s_\oA$ indicates that we only consider 
the differential forms
$\oE\in\Omega^*(M,\Algebroidg)$ which satisfy $s_\oE=s_\oA$, i.e.
the source maps associated with $\oA$ and $\oE$ are the same. 
Correspondingly, we obtain the cohomology groups associated with 
$\oA$ which are denoted by $\HG^*(M,\oA)$. Similarly, we can define 
the cohomology groups $\HQ^*(M,\qA)$ for every 
$\qA\in\CM(M,\Algebroidq_k)$.

\begin{thm}\label{thm:cohomology}
The cohomology groups $\HG^*(M,\oA)$ and $\HQ^*(M,\qA)$
are independent of the choice of $\oA$ and $\qA$ in their
gauge equivalence classes and their isomorphism classes are 
well-defined for every singular foliation 
\[\Fol\in \Mod(M,\Algebroidg_k,\Groupoidg_k)
\ \ \ \text{or} \ \ \ \Fol\in \Mod(M,\Algebroidq_k,\Groupoidq_k)
.\]
If $\Fol$ is non-singular, the cohomology groups 
are both isomorphic to the Bott cohomology of $\Fol$. 
\end{thm}  

The investigations of this paper indicate that the moduli spaces 
$\Mod(M,\Algebroidg_k,\Groupoidg_k)$ and 
$\Mod(M,\Algebroidq_k,\Groupoidq_k)$
share many properties as {\emph{completions}} of the space of 
all smooth framed foliations on $M$.  
The notion of concordance for germ cords and quantum cords, 
gives the quotients $\Ccal(M,\Algebroidg_k)$ and 
$\Ccal(M,\Algebroidq_k)$ which may be studied through 
the classification spaces $B\Gamma^\sg_k$ and $B\Gamma^\sq_k$ of 
$\Gamma^\sg_k$ and $\Gamma^\sq_k$, respectively.
The results of this paper would thus suggest that the space 
$[M,B\Gamma^\sq_k]$ of homotopy classes of maps from $M$ to 
$B\Gamma^\sq_k$ may be studied to classify  smooth foliations on 
$M$, similar to \cite{Haefliger-analytic}, 
 \cite{Mather-on-Haefliger} and  \cite{Thurston-classification}.\\

We start by the study of the gauge action of $\Diff^+(\R)$ on 
$C^\infty(\R)$-valued flat $1$-forms (cords) in 
Section~\ref{sec:flat-connections}. In Section~\ref{sec:groupoids} w
we introduce germ cords and quantum cords and 
their relevance in the study of codimension-one foliations.
Section~\ref{sec:holonomy} is devoted to the study of holonomy for 
leaf-like submanifold. In Section~\ref{sec:classification} we 
formulate and prove our classification theorems for germ and 
quantum cords and the corresponding classification of 
codimension-one foliations. Section~\ref{sec:cohomology} is a 
quick review of the relation between the cohomology theory for 
germ cords and quantum cords and the Bott cohomology of a foliation. 
Finally, Section~\ref{sec:general-case} states the main results 
of the previous  sections, which are formulated for foliations 
of codimension one, to the  case of general smooth framed 
foliations of arbitrary codimension.
 In \cite{AE-2}, we investigate complex cords and residues. The gauge 
 theoretic approach conveyed here takes us to a conjecture which is 
 both an attempt in fixing the Seifert conjecture~\cite{Seifert} 
 as well as a complex analogue to the 
 Novikov's compact leaf theorem~\cite{Novikov}.

\section{Flat connections and codimension-one foliations}
\label{sec:flat-connections}
\subsection{Diffeomorphisms of $\R$  as a gauge group}
Let ${\Diff^+(\R)}$  denote the group of orientation 
preserving diffeomorphisms of $\R$.
The Lie algebra  of ${\Diff^+(\R)}$ consists of 
smooth vector fields on $\R$ with the usual Lie bracket on 
vector fields, and is thus identified with ${C^\infty(\R)}$.  
Both $\Diff^+(\R)$ and $C^\infty(\R)$ are Fr\'echet spaces.
The Lie bracket on 
${C^\infty(\R)}$ is  given by
\begin{align*}
[\cdot,\cdot]:{C^\infty(\R)}\times{C^\infty(\R)}\ra 
{C^\infty(\R)},\ \ \ \ 
[\oA,\oB]:=\oA\oB'-\oA'\oB,\ \ \forall\ \oA,\oB\in{C^\infty(\R)}. 
\end{align*}  


Let us assume that $M$ is a smooth manifold of dimension $n$.
We may consider the smooth $C^\infty(\R)$-valued differential forms
\[\Omega^*(M,C^\infty(\R))=\bigoplus_{k=0}^{n}
\Omega^k(M,C^\infty(\R))=\bigoplus_{k=0}^{n}
\Gamma\left(M,\wedge^k(M)\otimes_\R C^\infty(\R)\right).\]
We denote  elements of $\Omega^k(M,C^\infty(\R))$ by capital letters 
in sans serif font, i.e. $\oA$, $\oB$. The differential of $\oA$ 
with respect to its $\R$-variable  is denoted by 
$\oA'$ or $\partial_\h\oA$. Note that $\oA'$ is also a smooth 
differential form with values in $C^\infty(\R)$. The Lie bracket of 
$C^\infty(\R)$ induces a Lie bracket on 
$\Omega^*(M,C^\infty(\R))$, giving it the structure of a
{\emph{differential graded Lie algebras}}, or a DG Lie algebra 
for short. Note that for $\oA\in\Omega^k(M,C^\infty(\R))$ and 
$\oB\in\Omega^l(M,C^\infty(\R))$
\begin{align*}
d[\oA,\oB]&=((d\oA)\oB'-(d\oA')\oB)+(-1)^k(\oA(d\oB')-\oA'(d\oB))
=[d\oA,\oB]+(-1)^k[\oA,d\oB].
\end{align*}
A $1$-form $\oA\in\Omega^1(M,C^\infty(\R))$ is called a 
{\emph{cord}} if it satisfies the Maurer-Cartan equation 
\[d\oA+\frac{1}{2}[\oA,\oA]=d\oA-\oA'\oA=0.\] 
The {\emph{space of cords}} is denoted by $\CM(M,{C^\infty(\R)})$.\\

The adjoint action of ${\Diff^+(\R)}$ on ${C^\infty(\R)}$ is 
described as follows.  For each $\ob\in{\Diff^+(\R)}$, the 
function $\Adj_\ob:{\Diff^+(\R)}\ra{\Diff^+(\R)}$ defined by 
\[\Adj_\ob(\oa)=\ob^{-1}\circ\oa\circ\ob\ \ \ \ 
\forall\ \oa\in{\Diff^+(\R)}\]  fixes the identity. 
Thus, the differential of $\Adj_\ob$ gives a smooth map
$d\Adj_\ob:{C^\infty(\R)}\ra {C^\infty(\R)}$ and the adjoint action
\begin{align*}
\star:{\Diff^+(\R)}\times{C^\infty(\R)}\ra {C^\infty(\R)},
\ \ \ \ \ \ \ob\star \oA:=d\Adj_\ob(\oA).
\end{align*} 
\begin{prop}\label{prop:group-action-general}
The adjoint action of ${\Diff^+(\R)}$ on 
${C^\infty(\R)}$ is given by
\begin{align}\label{eq:group-action}
\ob\star \oA:=\frac{\oA\circ \ob}{\ob'}\ \ \ \ \ 
\forall\ \ob\in{\Diff^+(\R)},\ \oA\in{C^\infty(\R)}.
\end{align}
\end{prop}
\begin{proof}
For $\oA\in {C^\infty(\R)}$, let us assume that $\oa_s$, for 
$s\in (-\epsilon,\epsilon)$, is a path in ${\Diff^+(\R)}$ with 
$\oa_0=Id_{\Diff^+(\R)}$ and $\partial_s\oa_s|_{s=0}=\oA$. The 
element $\ob\star \oA\in{C^\infty(\R)}$ is then given by
\begin{align*}
\partial_s\left(\ob^{-1}\circ \oa_s\circ \ob\right)_{s=0}
=\Big(\left((\ob^{-1})'\circ \oa_s\circ \ob\right)
\cdot \left(\partial_s\oa_s\right)\circ \ob\Big)_{s=0}
=\left((\ob^{-1})'\circ \ob\right)\cdot\left(\oA\circ\ob\right)
=\frac{\oA\circ \ob}{\ob'}.
\end{align*}
This completes the proof.
\end{proof}

Let us assume that $M$ is an oriented smooth manifold.   If 
$\ob\in\Omega^0(M,{\Diff^+(\R)})$ is a function on $M$ with values 
in ${\Diff^+(\R)}$ and $\oE\in\Omega^*(M,{C^\infty(\R)})$ is a 
${C^\infty(\R)}$-valued differential form on $M$, we can define 
$\oE\circ\ob\in\Omega^*(M,{C^\infty(\R)})$ by
\[(\oE\circ\ob)(\h,x):=\oE(\ob(\h,x),x)\ \ \ 
\forall\ \h\in \R,\ x\in M.\]
Note that 
\begin{align*}
(\oE\circ\ob)'=(\oE'\circ\ob)\ob'\ \ \ \ \text{and}\ \ \ 
d(\oE\circ\ob)=(d\oE)\circ\ob+(d\ob)(\oE'\circ\ob).
\end{align*}

\begin{prop}\label{prop:gauge-group}
The group $\Omega^0(M,\Diff^+(\R))$ acts on 
$\CM(M,{C^\infty(\R)})$  by
 \begin{align*}
&\star:\Omega^0(M,{\Diff^+(\R)})\times \CM(M,{C^\infty(\R)})
\ra \CM(M,{C^\infty(\R)}),\ \ \ \ \ \ \ \ 
\ob\star \oA:=\frac{\oA\circ \ob-d\ob}{\ob'}.
\end{align*}   
\end{prop}
\begin{proof}
Note that the action 
defined above is clearly smooth and that
\begin{align*}
\oc\star\left(\ob\star\oA\right)
&=\frac{\oA\circ\ob\circ\oc-(d\ob)\circ\oc
-(\ob'\circ\oc)d\oc}{(\ob'\circ\oc)\oc'}
=\frac{\oA\circ\ob\circ\oc-d(\ob\circ\oc)}{(\ob\circ\oc)'}
=\left(\ob\circ\oc\right)\star\oA.
\end{align*}
Let us assume that $\oA\in\Omega^1(M,C^\infty(\R))$,
$\oB=\oY\star\oA$ and $F_\oA=d\oA+\oA\oA'$. We then have
\begin{align*}
F_\oB=d\oB+\oB\oB'&=\frac{(dA)\circ\oY-(\oA'\circ\oY)d\oY
-(d\oY')\oB}{\oY'}
-\frac{\left((\oA'\circ\oY)\oY'-d\oY'-\oY''\oB\right)\oB}{\oY'}\\
&=\frac{(dA)\circ\oY-(\oA'\circ\oY)d\oY-(\oA'\circ\oY)(\oA\circ\oY)+
(\oA'\circ\oY)d\oY}{\oY'}=\frac{F_\oA\circ \oY}{\oY'}.
\end{align*}
In particular, if $\oA\in\CM(M,C^\infty(\R))$ then 
$\oY\star\oA\in \CM(M,C^\infty(\R))$, which completes the proof. 
\end{proof}

\subsection{Classification of cords}
\label{subsec:geometric}
One would naturally lean to study the quotient of the space of cords 
on a smooth manifold $M$ under the gauge action of  
$\Omega^0(M,\Diff^+(\R))$. 
It is expected that this quotient is identified with the 
conjugacy classes of group homomorphisms from  $\pi_1(M)$ to 
$\Diff^+(\R)$. Nevertheless, this statement, which comes to mind 
from working with finite dimensional Lie groups, is in general 
not true for infinite dimensional Lie groups (and their Lie 
algebras). The main issue is that the quotient 
$\CM(U,{C^\infty(\R)})/\Omega^0(U,\Diff^+(\R))$
can be non-zero, and in fact highly non-trivial, for small 
contractible open subsets $U\subset \R^n$, i.e. the 
{\emph{Poincar\'e Lemma}} is not satisfied.\\

We may easily construct $\oA\in\CM(U,C^\infty(\R))$ 
so that $\oA$ is not gauge equivalent to zero, i.e. 
so that $\oA$ is not of the form $-d\oY/\oY'$, 
even when $U$ is an open subset of a $1$-dimensional manifold 
(and may thus be identified with $\R$). For instance, let us  set
\[\oA(\h,s)=(1+\h^2)ds,
\ \ \ \ \forall\ (\h,s)\in\R\times \R.\]
If $\oA=-d\oY/\oY'$, it follows that 
$(1+\h^2)\partial_\h\oY=\partial_s\oY$. If we write 
$\h=\tan(\theta)$ it follows that 
$\partial_\theta\oY+\partial_s\oY=0$. In particular, $\oY$ is 
constant on the images of the curves 
\[\gamma_s:(-\pi/2,\pi/2)\ra \R^2  \ \ \ 
\gamma_s(\theta)=(\tan(\theta),s+\theta).\]
Nevertheless, this means that $\oY(\h,s+\pi_2)$ is 
bounded above by $\oY(0,s)$, which is a contradiction.

\begin{defn}\label{defn:transversely-trivial}
A cord $\oA\in\CM(M,{C^\infty(\R)})$ is called 
{\emph{locally trivial}} if for every point $x\in M$ there is an 
open subset $U\subset M$ containing $x$ so that $\oA|_U$ is gauge
equivalent to zero. The space of all locally trivial cords 
is denoted by $\Cordoba$. 
\end{defn}

It is then clear that $\Omega^0(M,\Diff^+(\R))$ takes locally 
trivial cords to locally trivial cords, and we thus obtain an 
action of $\Omega^0(M,\Diff^+(\R))$ on $\Cordoba$. 
Let us assume that $f:M_1\ra M_2$ is a smooth map
between smooth manifolds.  
If $\oA$ is locally trivial, it follows that $f^*\oA$
is also locally trivial. This gives a natural 
pull-back map \[f^*:\Cordoban2\ra \Cordoban1.\] 
Let us denote $\Cordoba/\Gauge$ by $\Moduli$.
From the equality 
$f^*(\oY\star \oA)=(f^*\oY)\star (f^*\oA)$, it follows that 
$f^*$ induces a map 
\[f^*:\Modulin2\ra \Modulin1.\]

\begin{thm}\label{thm:classification-cords}
For every smooth manifold $M$, there is a 
natural bijection
\[\rho_{M}:\Moduli\lra 
\Hom(\pi_1(M),\Diff^+(\R))/\Diff^+(\R),\] 
where $\Diff^+(\R)$ acts on $\Hom(\pi_1(M),\Diff^+(\R))$ by 
conjugation. 
If $f:M_1\ra M_2$ is a smooth map between smooth 
manifolds, the following diagram is commutative:
\begin{align*}
\begin{diagram}
\Modulin2&\rTo{f^*}& 
\Modulin1\\
\dTo{\rho_{M_2}}&&\dTo_{\rho_{M_1}}\\
\Hom(\pi_1(M_2),\Diff^+(\R))/\Diff^+(\R)&\rTo{\ \ \ f^*\ \ \ }&
\Hom(\pi_1(M_1),\Diff^+(\R))/\Diff^+(\R)
\end{diagram}
\end{align*}
\end{thm}
\begin{proof}
Let us assume that $\oA\in\Cordoba$. $M$ may then be covered
by the open subsets $U_\alpha$ so that $\oA|_{U_\alpha}$ is 
gauge equivalent to zero. Choose 
$\oY_\alpha\in\Omega^0(U_\alpha,\Diff^+(\R))$ so that 
$\oA|_{U_\alpha}=\oY_\alpha\star 0$. Over 
$U_\alpha\cap U_\beta$ we find
\[\oA|_{U_\alpha\cap U_\beta}=\oY_\alpha\star 0=\oY_\beta\star 0
\ \ \ \Rightarrow \ \ \ (\oY_\alpha\circ\oY_\beta^{-1})\star 0=0. \]
This means that $c_{\alpha\beta}=\oY_\alpha\circ\oY_\beta^{-1}
:U_\alpha\cap U_\beta\ra \Diff^+(\R)$ is locally constant.
These maps satisfy the cocycle condition and give a cohomology class
in the \v{C}ech cohomology  $\HC^1(M,\Diff^+(\R))$.
The functions $\oY_\alpha$ are well-defined only upto composition
with locally constant functions. If 
$\oZ_\alpha=d_\alpha\circ \oY_\alpha$, where 
$d_\alpha:U_\alpha\ra \Diff^+(\R)$ is locally constant,  
the \v{C}ech cocycle $c'_{\beta\alpha}$ associated with 
$\{\oZ_\alpha\}$ would be given by
\[c'_{\alpha\beta}=(d_\alpha\circ \oY_\alpha)\circ 
(d_\beta\circ \oY_\beta)^{-1}
=d_\alpha\circ c_{\alpha\beta} \circ d_\beta^{-1}
:U_\alpha\cap U_\beta\ra \Diff^+(\R).\]
This means that the cohomology classes represented by 
$\{c_{\alpha\beta}\}$ and $\{c'_{\alpha\beta}\}$ are the same.
Moreover, if we gauge $\oA$ by $\oY\in\Omega^0(M,\Diff^+(\R))$, the 
cocycles associated with $\oA$ and $\oY\star \oA$ are the same 
and we obtain a well-defined map
\begin{align*}
\rho=\rho_{M}:\Moduli\ra 
\HC^1(M,\Diff^+(\R))\simeq \Hom(\pi_1(M),\Diff^+(\R))/\Diff^+(\R).
\end{align*}
If $\rho(\oA)=\rho(\oB)$ it follows that the corresponding cocycles
$\{c^\oA_{\alpha\beta}\}$ and $\{c^\oB_{\alpha\beta}\}$
are related by the locally constant functions 
$d_\alpha:U_\alpha\ra \Diff^+(\R)$, after we pass to refinement 
of the coverings associated with $\oA$ and $\oB$. In particular,
\begin{align*}
d_\alpha\circ c_{\alpha\beta}^\oA=c_{\alpha\beta}^\oB\circ d_\beta
\ \ \ \Rightarrow \ \ \ 
d_\alpha\circ \oY_\alpha\circ\oY_\beta^{-1}=
\oZ_\alpha\circ\oZ_\beta^{-1}\circ d_\beta,
\end{align*}
where $\oZ_\alpha\star 0=\oB|_{U_\alpha}$. It follows from the above 
equality that $\oZ_\beta^{-1}\circ d_\beta\circ \oY_\beta=
\oZ_\alpha^{-1}\circ d_\alpha\circ \oY_\alpha$ over 
$U_\alpha\cap U_\beta$, and we can define
$\oX\in\Omega^0(M,\Diff^+(\R))$ by 
$\oX|_{U_\alpha}=\oZ_\alpha^{-1}\circ d_\alpha\circ \oY_\alpha$.
Over $U_\alpha$ we can then compute
\begin{align*}
\oX\star \oB|_{U_\alpha}=
\left(\oZ_\alpha^{-1}\circ d_\alpha\circ \oY_\alpha\right)\star
\left(\oZ_\alpha\star 0\right)=\oY_\alpha\star d_\alpha\star 0
=\oA|_{U_\alpha}.
\end{align*}
Thus, $\oA=\oX\star \oB$ and the two cords are gauge equivalent.
It follows from here that over simply connected manifolds every 
locally trivial cord is gauge equivalent to zero.\\

To finish the proof, we need to show that the map 
$\rho=\rho_{M}$ is surjective.
Let us assume that 
$\{c_{\alpha\beta}:U_\alpha\cap U_\beta\ra \Diff^+(\R)\}$ 
is a cocycle in $\HC^1(M,\Diff^+(\R))$. Choose a smooth partition 
of unity $\{\lambda_\alpha:U_\alpha\ra \R^{\geq 0}\}_\alpha$
subordinate to the cover $\{U_\alpha\}_\alpha$ of $M$ and  define
\begin{align*}
\oY_\alpha:U_\alpha\ra \Diff^+(\R),\ \ \ \oY_\alpha(\h,x)&=
\sum_{\gamma} \lambda_\gamma(x) c_{\gamma\alpha}(\h).
\end{align*}
Note that $\partial_t\oY_\alpha>0$, and $\oY_\alpha(\cdot,x)$ 
is thus a diffeomorphism for all $x\in U_\alpha$. 
Over the intersections $U_\alpha\cap U_\beta$ we have
$\oY_\alpha\circ c_{\alpha\beta}=\oY_\beta$.
If we define $\oZ_\alpha:U_\alpha\ra \Diff^+(\R)$ by 
$\oZ(x,\cdot)=\oY(x,\cdot)^{-1}$, we find
$\oZ_\beta=c_{\beta\alpha}\circ \oZ_\alpha$ and thus
$\oZ_\beta\star 0=\oZ_\beta\star c_{\alpha\beta}\star 0=
(c_{\alpha\beta}\circ \oZ_\beta)\star 0=\oZ_\alpha\star 0$.
In particular, the cords $\oZ_\alpha\star 0\in
\CMTT(U_\alpha, C^\infty(\R))$ may be glued together to give 
$\oA\in \Cordoba$. It is also clear that 
$\rho(\oA)=[\{c_{\alpha\beta}\}]\in \HC^1(M,\Diff^+(\R))$.
\end{proof}

\section{Lie groupoids, Lie algebroids and gauge theory}
\label{sec:groupoids}
\subsection{Lie groups and Lie algebroids}
A {\emph{topological groupoid}} is a small topological category 
$\Groupoid$ such that all arrows are invertible. This means that 
the sets $\Obj(\Groupoid)$ and $\Mor(\Groupoid)$ of objects 
and morphisms of $\Groupoid$ are topological spaces and there 
are maps 
\[s,t:\Mor(\Groupoid)\ra \Obj(\Groupoid),\ \ \ e:\Obj(\Groupoid)\ra 
\Mor(\Groupoid)\]
which assign the source $x=s(\phi)$ and the target $y=t(\phi)$ 
to a morphism $\phi\in \Mor(x,y)$, and the identity map 
$e(x)\in\Mor(x,x)$ for $x,y\in\Obj(\Groupoid)$.  These maps 
are required to be continuous and have 
the appropriate properties, e.g. that 
$\phi\circ e(x)=\phi=e(y)\circ \phi$ for 
every $\phi\in\Mor(x,y)$. Moreover, we require that the arrows 
are invertible, meaning that for every $\phi\in\Mor(x,y)$ there 
is a unique morphism $\phi^{-1}\in\Mor(y,x)$ such that 
$\phi\circ\phi^{-1}=e(y)$ and $\phi^{-1}\circ\phi=e(x)$. 
Let us review some important examples which play  important role 
in this paper.\\

{\bf{The covering groupoid.}} Let 
$\Ucal=\{U_\alpha\}_{\alpha\in I}$ be a 
covering of a smooth manifold $M$ with open subsets.
 We can then define a groupoid $\Gamma_\Ucal$ by setting 
\[\Obj(\Gamma_\Ucal)=\sqcup_{\alpha\in I}U_\alpha,\ \ \ 
\Mor(\Gamma_\Ucal)=\{(x,\alpha,\beta)\ |\ \alpha,\beta\in I, 
x\in U_\alpha\cap U_\beta\}.\]
The morphism $(x,\alpha,\beta)$ is usually denoted by 
$(x\in U_\beta)\xrightarrow{(x,\alpha,\beta)} (x\in U_\alpha)$.
The source map sends $(x,\alpha,\beta)$ to $(x\in U_\beta)$ and the 
target map sends it to $(x\in U_\alpha)$. The composition is defined 
by setting $(x,\alpha,\beta)\star(x,\beta,\gamma)=(x,\alpha,\gamma)$.
This gives a groupoid which is sometimes denoted by $\Gamma_\Ucal$ 
in the literature.\\

{\bf{Germs of diffeomorphisms.}} Let us assume that $R$ 
is a smooth oriented manifold. Let $\Diff^+(R)$ denote the space 
of orientation preserving diffeomorphisms of $R$. We can define 
the groupoid of germs of local diffeomorphisms of $R$, denoted 
by $\Gamma^{\mathsf{g}}_R$, by setting
$\Obj(\Gamma^{\mathsf{g}}_R)=R$ and 
\begin{align*}
\Mor(\Gamma^{\mathsf{g}}_R)=\bigsqcup_{x,y\in R}\Morg(x,y)\ \ \ \  
\text{where}\ \ \ &\Morg(x,y)=\{(f,x,y)\ 
|\ f\in \Diff^+(R), f(x)=y\}/\simg
\end{align*}
and $(f,x,y)\simg (g,x,y)$ if there is an open neighborhood $U$ 
of $x$ in $R$ so that $f|_U=g|_U$. The equivalence class of 
$(f,x,y)$ is denoted by $[f,x,y]_{\sg}$ or $[f,x]_{\sg}$ 
(note that $y=f(x)$ is determined by $f$ and $x$).
The source map and the target map are then defined
by $s[f,x,y]_{\sg}=x$ and $t[f,x,y]_{\sg}=y$, while the map 
$e$ is defined by $e(x)=[Id_M,x,x]_{\sg}\in\Morg(x,x)$. 
The topology on $\Gamma^{\mathsf{g}}_R$ 
is defined so that a basis of neighborhoods for $[f,x]_{\sg}$ 
is given by $\{[f,y]_{\sg}\ |\ y\in U\}$, where $U$ is an open 
set in $R$ which contains $x$. Note that the topology induced 
on $\Morg(x,x)$ is the discrete topology. With this topology 
$\Mor(\Gamma^{\mathsf{g}}_R)$ has the structure of a manifold 
which is equipped with the covering maps $s$ and $t$ to $R$.\\

{\bf{Quantum groupoid of diffeomorphisms.}} Let us continue 
to assume that $R$ is a smooth oriented manifold.  We can define 
the groupoid of quantum  diffeomorphisms of $R$, denoted by 
$\Gamma^{\mathsf{q}}_R$, by setting
$\Obj(\Gamma^{\mathsf{q}}_R)=R$ and 
\begin{align*}
\Mor(\Gamma^{\mathsf{q}}_R)=\bigsqcup_{x,y\in R}\Morq(x,y)\ \ 
\text{where}\ \ 
\Morq(x,y)=\{(f,x,y)\ |\ f\in \Diff^+(R), f(x)=y\}/\simq
\end{align*}
and $(f,x,y)\simq (g,x,y)$ if the Taylor expansions of $f$ and $g$ 
agree at $x$. The equivalence class of $(f,x,y)$ is denoted by 
$[f,x,y]_{\sq}$ or $[f,x]_{\sq}$.
The source map and the target map are defined
by $s[f,x,y]_\sq=x$ and $t[f,x,y]_\sq=y$, while the map $e$ 
is defined by 
$e(x)=[Id_M,x,x]_\sq\in\Morq(x,x)$. The topology on 
$\Gamma^{\mathsf{q}}_R$ is defined 
so that a basis of neighborhoods for $[f,x]_\sq$ is given by 
$\{[f,y]_{\sq}\ |\ y\in U\}$, where $U$ is an open set in $R$ 
which contains $x$. As before, note that the topology induced 
on $\Morq(x,x)$ is the discrete topology. Unlike 
$\Gamma^{\mathsf{g}}_R$, the groupoid $\Gamma^{\mathsf{q}}_R$ 
is not Haussdorff, as different germs may have the same Taylor
 expansion.\\

The particular cases where the manifold $R$ is the real line 
$\R$ is of particular interest. In this case, we write 
$\Gamma^\sg=\Gamma^\sg_\R$ and
$\Gamma^\sq=\Gamma^\sq_\R$. We write $\Gamma^\sg_k$ for 
$\Gamma^\sg_{\R^k}$ and $\Gamma^\sq_k$ for $\Gamma^\sq_{\R^k}$.
There is a well-defined 
{\emph{Taylor expansion functor}} 
$\quot:\Gamma^\sg_k\ra \Gamma^\sq_k$ which is the identity map 
over the objects and takes $[f,x]_\sg$ to $[f,x]_\sq$. This 
functor is a homomorphism of groupoids.\\

The topology on $\Gamma^\sg$ and $\Gamma^\sq$ is not so pleasant
if one would like to treat them as Lie groupoids. The arrows of 
the groupoid $\Gamma^\sq$, which are given by the formal power 
series $\sum_{m=0}^\infty y_m(\h-x)^m$  may naturally be 
identified with the points 
\[(x,y_0,y_1,y_2,\ldots)\in\R\times\R\times\R^+\times \R^\infty.\]
This gives a natural candidate for equipping $\Mor(\Gamma^\sq)$ 
with a topology and arriving at a Lie groupoid, i.e. 
a topological groupoid with the 
structure of a manifold on the spaces of objects and morphisms.
 We denote this Lie groupoid by 
$\Groupoidq$. Note that $\Groupoidq$ and $\Gamma^\sq$ are the same 
as groupoids, but not as topological groupoids. 
The Lie algebroid associated with $\Groupoidq$ is the 
vector bundle $\Algebroidq$ over $\R$ which is given by 
\[\Algebroidq_{s}=\left\{\sum_{m=0}^\infty a_m(\h-s)^m\ \big|\ 
a_m\in\R, m=0,1,2,\ldots\right\}\ \ \ \ \ \ \forall\ s\in\R.\] 
A section $\qA=\sum_{m=0}^\infty a_m(s)(\h-s)^m$ 
of the vector bundle $\Algebroidq$ is determined by the 
smooth functions $a_m:\R\ra \R$ for $m\in\Z^{\geq 0}$. \\

Equipping  $\Gamma^\sg$ with a more appropriate topology is 
difficult. In fact, it is a common belief that it is not possible 
to construct a "{\emph{good}}" non-discrete topology 
on $\Morg(x,x)$. For instance, Gromov \cite{Gromov-book} writes: 
"{\emph{there is no useful topology in this space ... of germs of 
$[C^k]$ sections...}}".  Instead of choosing a topology on 
$\Gamma^\sg$, it is enough for to define the notion of a smooth 
function with values in $\Gamma^\sg$. 
A map $\oA:M\ra \Gamma^\sg$ from a smooth manifold to 
$\Gamma^\sg$  is called {\emph{smooth}} if for every $x\in M$ 
there is an open set $U_x\subset M$ containing $x$, an 
open interval $I\subset \R$ and a smooth function 
$\ov\oA:I\times U_x\ra \R$ 
such that $\oA(y)$ is represented by $\ov\oA(\cdot,y):\R\ra\R$ 
for every $y\in U_x$. We write $\Groupoidg$ for $\Gamma^\sg$ if this 
weak notion of topology is used. It then makes sense to talk 
about the tangent bundles $T\Groupoidg$ of the space of arrows.  
The Lie algebroid $\Algebroidg$ may thus be defined where for 
$s\in\R$, the vector space $\Algebroidg_{s}$ consists of germs 
of smooth real-valued functions at $s$. 
The derivatives of $\oA$ and $\qA$ with respect to the variable $\h$ 
are denoted $\oA'$ and $\qA'$, for simplicity.\\

The homomorphism $\quot:\Gamma^\sg\ra \Gamma^\sq$ may also 
be regarded as a homomorphism $\quot:\Groupoidg\ra\Groupoidq$ 
and induces a homomorphism of vector bundles 
$\quot:\Algebroidg\ra\Algebroidq$. 

\subsection{Germ cords, quantum cords and the gauge actions}
Let us assume that $M$ is a smooth manifold. A germ $k$-form on 
$M$ is a smooth $k$-form $\oA$ on $M$ with values in $\Algebroidg$.
 The smooth map 
$s\circ\oA$ then induces a well-defined smooth map 
$s_\oA:M\ra \R$, which is called the {\emph{source}} of $\oA$.
The  space of all germ $k$-forms is denoted by 
$\Omega^k(M,\Algebroidg)$.
Similarly, we can define $\Omega^k(M,\Algebroidq)$.
The source map defines the source maps
\[s:\Omega^*(M,\Algebroidq)\ra C^\infty(M)\ \ \ \text{and}\ \ \ 
s:\Omega^*(M,\Algebroidg)\ra C^\infty(M),\]
while the target map induces  the maps
\[t:\Omega^*(M,\Algebroidq)\ra \Omega^*(M,\R)\ \ \ \text{and}\ \ \ 
t:\Omega^*(M,\Algebroidg)\ra \Omega^*(M,\R).\]
Denoting both source maps by $s$  and both target maps by $t$ 
is of course an abuse of notation,
which will be repeated in many similar situations in this paper. The 
homomorphisms $\quot:\Gamma^\sg\ra\Gamma^\sq$ induces 
a {\emph{Taylor expansions}} map
\[\quot:\Omega^*(M,\Algebroidg)\ra \Omega^*(M,\Algebroidq).\]

The Lie brackets on the fibers of $\Algebroidg$ and $\Algebroidq$ 
(which is induced by the Lie bracket of these algebroids) induces  
Lie brackets on $\Omega^*(M,\Algebroidg)$ and 
$\Omega^*(M,\Algebroidq)$, and are defined only when the source 
maps match. The Lie bracket is given by 
\begin{align*}
&[\oA,\oB]:=\oA\oB'-\oB\oA'\ \ \ \ \ 
\ \ \ \forall\ \oA\in\Omega^k(M,\Algebroidg),
\oB\in\Omega^l(M,\Algebroidg)
\ \ \text{s.t.}\ \ s_\oA=s_\oB:M\ra \R.
\end{align*}
A formula for the  Lie bracket of $\Omega^*(M,\Algebroidq)$ is 
given similarly.

\begin{defn} A {\emph{germ cord}} is a $1$-form 
$\oA\in\Omega^1(M,\Algebroidg)$ which satisfies 
\[d\oA+\frac{1}{2}[\oA,\oA]=d\oA+\oA\oA'=0.\]  
The space of germ cords  on $M$ is denoted by $\CM(M,\Algebroidg)$. 
\end{defn}

 The gauge action of $\Groupoidg$  over 
$\CM(M,\Algebroidg)$ is defined 
as follows. We define 
\begin{equation}\label{eq:gauge-action}
\oY\star \oA:=\frac{\oA\circ\oY-d\oY}{\oY'}\ \ \ \ \ 
\forall\ \ \oA\in\CM(M,\Algebroidg),\ 
\oY\in\Omega^0(M,\Groupoidg).
\end{equation}
As usual, it is understood from the definition that 
$\oY\star \oA$ makes sense 
only if $t_\oY=s_\oA$ as smooth functions from $M$ to $\R$. Setting 
$\oB=\oY\star\oA$ and $F_\oA=d\oA+\oA\oA'$, we have
$F_\oB=(F_\oA\circ \oY)/\oY'$.
In particular, if $\oA\in\CM(M,\Algebroidg)$ then 
$\oY\star\oA\in\CM(M,\Algebroidg)$.
 If $\oY,\oZ\in\Omega^0(M,\Groupoidg)$ 
are gauge maps with $t_\oZ=s_\oY$, we also compute
$\oZ\star\left(\oY\star\oA\right)
=\left(\oY\circ\oZ\right)\star\oA$.
These observations show that $\star$ defines an action of 
$\Omega^0(M,\Groupoidg)$ on $\Omega^1(M,\Algebroidg)$ which 
induces and action on $\CM(M,\Algebroidg)$,
called the  {\emph{gauge action}} of $\Groupoidg$ on 
$\CM(M,\Algebroidg)$. The gauge 
action of $\Groupoidq$ on $\Omega^1(M,\Algebroidq)$ is 
defined in a similar 
way, as Equation~\ref{eq:gauge-action} makes 
sense for the formal power series. If 
$\qA=\sum_{m=0}^\infty a_m(\h-s)^m$ then 
\begin{align*}
&F_\qA=\sum_{m=0}^\infty\left(da_m+(m+1)a_{m+1}ds\right)(\h-s)^m
-\sum_{p,q}(p-q)a_pa_q(\h-s)^{p+q-1}\\
\Rightarrow\ \ \ \ \ \ &F_\qA=0\ \ \iff\ \ 
da_m+(m+1)a_{m+1}ds=\sum_{p+q=m+1}(p-q)a_pa_q\ \ \ m=0,1,\cdots
\end{align*}

\begin{defn} 
A {\emph{quantum cord}} is a $1$-form 
$\qA\in\Omega^1(M,\Algebroidg)$ which  satisfies 
$F_\qA=d\qA+\qA\qA'=0$ and is locally trivial, meaning that
for every point $x\in M$ there is an open set $U_x\subset M$
containing $x$ so that $\qA|_{U_x}$ is of the form 
$\qY\star 0$ for some $\qY\in\Omega^0(U_x,\Groupoidq)$. The space 
of quantum cords on $M$ is denoted by $\CM(M,\Algebroidq)$. 
\end{defn} 
We will soon see the reason for the extra local triviality 
condition in the case of quantum cords. The action of 
$\Omega^0(M,\Groupoidq)$ on $\Omega^1(M,\Algebroidq)$ 
induces a gauge action of $\Groupoidq$ on $\CM(M,\Algebroidq)$.
Note that $\quot$ induces a well-defined homomorphism 
$\quot:\Omega^*(M,\Algebroidg)\ra \Omega^*(M,\Algebroidq)$ 
which restricts to 
$\quot:\CM(M,\Algebroidg)\ra \CM(M,\Algebroidq)$. This follows 
from Lemma~\ref{lem:Poincare-Lemma}, which will be proved in 
Section~\ref{sec:classification}.

\subsection{Foliations of codimension one}
Let us denote the fibers of $\Algebroidg$ and $\Algebroidq$ over
$0\in\R$ by $\algebroidg$ and $\algebroidq$.
Denote the group of local diffeomorphisms of $(\R,0)$, 
which consists of the arrows in $\Groupoidg$ with source 
and target equal to $0\in\R$,
by $\Groupg$. Similarly, let $\Groupq$ denote the group of 
power series
$\sum_{m=1}^\infty a_m\h^m$ with $a_1>0$, which consists of the 
arrows in $\Groupoidq$ with source and target equal to $0$. 
It is then clear that $\Groupg$ and $\Groupq$ are both groups.
From the gauge action of the groupoids $\Groupoidg$ and $\Groupoidq$
on $\CM(M,\Algebroidg)$ and $\CM(M,\Algebroidq)$ we obtain 
the gauge actions of  $\Groupg$ and $\Groupq$ on 
\[\CM(M,\algebroidg)\subset\Omega^1(M,\algebroidg)\ \ \ \text{and}
\ \ \ \CM(M,\algebroidq)\subset\Omega^1(M,\algebroidq),\]
respectively.
The target maps 
\[t:\Omega^1(M,\Algebroidg)\ra \Omega^1(M,\R)
\ \ \ \text{and}\ \ \ t:\Omega^1(M,\Algebroidq)\ra \Omega^1(M,\R)\]
 may be used to associate 
a $1$-form $a_0$ to every germ cord $\oA\in \CM(M,\algebroidg)$, 
or its quantized image $\qA=\quot(\oA)$. In fact, 
$\qA=\sum_{m=0}^\infty a_m\h^m$, and 
$a_0$ is the initial term in this Taylor expansion. It follows 
that $da_0=a_1\wedge a_0$. If we further assume that $a_0$ is 
nowhere zero, it follows that $a_0$ determines a smooth 
transversely oriented codimension-one foliation on $M$. Let us 
denote the subspaces of $\algebroidg$ and $\algebroidq$ 
which consists of the elements which are positive at the origin by 
$\algebroidg^*$ and $\algebroidq^*$, respectively. 
Correspondingly,  the subsets 
\[\CM(M,{\algebroidg^*})\subset \CM(M,\algebroidg)\ \ \ 
\text{and}\ \ \ \CM(M,{\algebroidq^*})\subset \CM(M,\algebroidq)\] 
of  {\emph{non-singular}} germ cords and quantum cords 
may be defined. 
The above discussion implies that there are projection maps
\begin{align*}
\proj_{\mathsf{g}}:\CM(M,{\algebroidg^*})\ra \Fcal(M)
\ \ \ \text{and}
\ \ \ \proj_{\mathsf{q}}:\CM(M,{\algebroidq^*})\ra \Fcal(M),
\end{align*}
where $\Fcal(M)$ denotes the space of smooth transversely 
oriented  codimension-one foliations on $M$.  
We abuse the notation and denote $\proj_{\mathsf{g}}$ 
and $\proj_{\mathsf{q}}$ by $\proj$.\\

Let us assume that $\Fol\in\Fcal(M)$ 
 is given by a $1$-form $a_0\in\Omega^1(M,\R)$ 
such that $da_0=a_1\wedge a_0$, for another $1$-form 
$a_1\in\Omega^1(M)$. There is a vector field 
$V$ transverse to $\Fol$ which satisfies $a_0(V)=-1$. 
By subtracting a suitable multiple of $a_0$ from $a_1$, we may 
further assume that $a_1(V)=0$. 
The vector field $V$ may be integrated to give 
the flow $\Phi_\h=\Phi^V_\h$ on $M$. For every $x\in M$ we can 
then define
$\oA=\oA_{a_0,V}\in\Omega^1(M,C^\infty(\R))$ by 
\begin{align*}
\oA(\h,x):=\Phi_\h^*(a_0)(x)\ \ \ \ \forall\ x\in M,\ \h\in\R.
\end{align*}
The $1$-form $\oA$ may  be considered as an element in 
$\Omega^1(M,\algebroidg)$. 
The Taylor expansion of $\oA$ is given as follows. Let $L$ 
denote the Lie derivative corresponding to $V$ and define 
\[\qA=e^{\h L}a_0=\sum_{n=0}^{\infty}\frac{L^n(a_0)\h^n}{n!}
=\sum_{n=0}^\infty a_n\h^n.\]
\begin{lem}
Having fixed the above notation,  $\oA_{a_0,V}$ is a cord and 
its image in $\Omega^1(M,\algebroidq)$ is a germ cord, while 
 $\qA_{a_0,V}=\quot(\oA)$ is quantum cord.
\end{lem} 
\begin{proof}
First note that
\begin{align*}
L(a_0)=d(\imath_V(a_0))+\imath_V(d(a_0))
=\imath_V(a_1a_0)
=a_1(V)a_0-a_0(V)a_1=a_1.
\end{align*}
We then observe that the derivative $\oA'$ of $\oA$ is given by
\[\oA'=\frac{d}{d\h}\Phi_\h^*(a_0)=\Phi_\h^*(L(a_0))=\Phi_\h^*(a_1).\]
It follows that
$d\oA+\oA\oA'=\Phi_\h^*(da_0+a_0a_1)=0$. 
It follows from Lemma~\ref{lem:locally-trivial} that 
$\qA=\quot(\oA)$ is a quantum cord.
\end{proof}
Note that $\imath_V\oA_{a_0,V}=-1$. The cord $\oA=\oA_{a_0,V}$ 
is uniquely specified by the following two conditions
\begin{itemize}
\item $\oA(0,x)=a_0(x)$ for every $x\in M$.
\item We have $\imath_V(\oA)=-1$.
\end{itemize}
We call $a_0=\oA(0,\cdot)$ the initial term of $\oA$.
More generally, let us assume that $\oa$ 
is a  $C^\infty(\R)$-valued function on $M$ with 
negative initial term. Consider the equation
\begin{align}\label{eq:recursion-A}
&L(\oB)+[\oB,\oa]+d\oa=0.
\end{align}
This is a differential equation for $\oB$ which defines 
$\oB$ (in a neighborhood of the origin in $\R$) 
once the initial term of $\oB$ is fixed.
In particular, if we set $\oB(0,x)=\oa(0,x)\oA(0,x)$ 
and solve for $\oB$, we obtain the $1$-form 
$\oB=\oA_{\Fol,V,\oa}\in\Omega^{1}(M,\algebroidg)$.
\begin{prop}\label{prop:MC-elements}
Fix the transversely oriented codimension-$1$ foliation $\Fol$  
and the transverse vector field $V$. The element 
$\oA_{\Fol,V,\oa}\in\Omega^1(M,\algebroidg)$  is the 
unique germ cord in  $\CM(M,\algebroidg)$ which satisfies 
the following two conditions.
\begin{itemize}
\item The initial term $\oA_{\Fol,V,\oa}(0,\cdot)$
of $\oA_{\Fol,V,\oa}$ defines the foliation $\Fol$.
\item The equation $\imath_V(\oA)+\oa=0$ is satisfied 
in $\Omega^0(M,\algebroidg)$.
\end{itemize}
\end{prop} 
\begin{proof}
Let $\oB=\oA_{\Fol,V,\oa}$. We first need to show that 
$d\oB=\oB'\oB$. Let us assume that 
\[\quot(\oB)=\sum_{n=0}^\infty b_n\h^n\ \ \ \text{and}\ \ \ 
\quot(\oa)=\sum_{n=0}^\infty x_n\h^n.\] It is 
then clear that $\oB'(0,x)=b_1(x)$ is given as 
$L(b_0)+x_1b_0+dx_0$ and we compute
\begin{align*}
b_1b_0&=(L(b_0)+x_1b_0+dx_0)a_0
=(x_0L(a_0)+dx_0)a_0
=d(x_0a_0).
\end{align*}
Thus, $d(b_0)=b_1b_0$ and the initial term of 
$\oE=d\oB-\oB'\oB$ is zero. Moreover, the differential equation 
of (\ref{eq:recursion-A}) implies $L(\oB')+d\oa'+B\oa''-B''\oa=0$ 
and we can thus compute
\begin{align*}
L(\oE)+[\oE,\oa]&=dL(\oB)-L(\oB')\oB-\oB'L(\oB)+
d\oB\oa'-\oB'\oB\oa'-d\oB'\oa-\oB''\oB\oa\\
&=dL(\oB)-L(\oB')\oB-\oB'L(\oB)+d\oB\oa'+\oB'(L(\oB)
+d\oa)-d\oB'\oa-\oB(L(\oB')+d\oa')\\
&=dL(\oB)+d\oB\oa'-\oB d\oa'-d\oB'\oa+\oB' d\oa\\
&=d\left(L(\oB)+\oB \oa'-\oB'\oa\right)=0
\end{align*} 
The equation $L(\oE)+[\oE,\oa]=0$ and the initial value condition
$\oE(0)=d(b_0)-b_1b_0=0$ uniquely determine $\oE$ in a 
neighborhood of the origin. This implies that in a neighborhood of 
the origin we have $\oE=0$, i.e. $\oB=\oA_{\Fol,V,\oa}$ is a germ 
cord and belongs to $\CM(M,\algebroidg)$.
Let us denote $\imath_V(\oB)$ by $\oC$. It follows that
\begin{align*}
L_V(\oB)+[\oB,\oa]+d\oa=0\ \ \ 
\Rightarrow\ \ \ &(\imath_V\circ d)(\oC)+[\imath_V(\oB),\oa]
+(\imath_V\circ d)(\oa)=0\\
\Rightarrow\ \ \ &L(\oC+\oa)+[\oC,\oC+\oa]=L(\oC+\oa)+[\oC,\oa]=0.
\end{align*}
From this last equation, and the uniqueness of solutions for 
differential equations, it follows that $X=-C$. \\

Let $\oA$ be a germ cord in $\CM(M,\algebroidg)$ which is 
compatible with a foliation $\Fol$. Set $\oa=-\imath_V(\oA)$ 
and let $\oB=\oA_{\Fol,V,\oa}$. By definition, $\oA(0,x)=a_0$ and 
$\oB(0,x)=b_0$ differ by multiplication by a non-zero constant. 
Since $\imath_V(a_0)=\imath_V(b_0)$, it follows that the initial
 terms of $\oA$ and $\oB$ agree. Moreover, both $\oA$ and 
$\oB$ satisfy Equation~\ref{eq:recursion-A}, and we thus have
$L(\oB-\oA)+[\oB-\oA,\oa]=0$. This differential equation for 
$\oB-\oA$, together with the initial condition that the initial
term of $\oB-\oA$ vanishes, force $\oB-\oA$ to vanish in a 
neighborhood of the origin in $\R$.
\end{proof}

Proposition~\ref{prop:MC-elements} implies that the germ cords
in $\CM(M,\algebroidg)$ which correspond to a foliation 
$\Fol$ are determined by 
their evaluation over the vector field $V$. This evaluation map 
takes its values in $\Omega^0(M,\algebroidg)$. 

\begin{remark}\label{rmk:quantum-version}
The same statement is also true for $\algebroidq$, that the quantum 
cords in $\CM(M,\algebroidq)$ which correspond to $\Fol$ 
are determined by their 
evaluation over the vector field $V$. This latter evaluation map 
 takes its values in $\Omega^0(M,\algebroidq)$. 
\end{remark} 
\begin{remark}\label{rmk:group-to-groupoid}
Fix $\oA\in\CM(M,\Algebroidg)$ and let
$\qA=\quot(\oA)=\sum_{m=0}^\infty
a_m(\h-s)^m$.
Note that $s=s_\oA=s_\qA$ is a smooth function while 
$a_i\in\Omega^1(M,\R)$. Since $\oA$ is a germ cord (and $\qA$ 
is a quantum cord) it follows 
that $da_0=a_1\wedge (a_0+ds_\oA)$. If we further assume that 
$a=a_0+ds_\oA\in\Omega^1(M,\R)$ is nowhere zero, 
it follows that $a$ 
determines a smooth transversely oriented codimension-one foliation 
on $M$. If $\oB=\oY\star \oA$ with $\oY(\h,x)=\h+s_\oA(x)$ 
(and  $s_\oB=0$) we find $b_0=a$ while  $\oB\in\CM(M,\algebroidg)$.
This observation implies that every germ cord (respectively, 
quantum cord) 
is gauge equivalent to a germ cord (respectively, quantum cord) 
in $\CM(M,\algebroidg)$ (respectively, in $\CM(M,\algebroidq)$).
The induced actions of 
$\Omega^0(M,\Groupg)$ and $\Omega^0(M,\Groupq)$ on 
$\CM(M,\algebroidg)$ and $\CM(M,\algebroidq)$ give the 
moduli spaces
\begin{align*}
\Mod(M,\algebroidg,\Groupg)&=
\CM(M,\algebroidg)/\Omega^0(M,\Groupg)\ \  \text{and}\ \
\Mod(M,\algebroidq,\Groupq)=
\CM(M,\algebroidq)/\Omega^0(M,\Groupq).
\end{align*}
The passage from $\algebroidg$ and $\algebroidq$ to 
$\Algebroidq$ and $\Algebroidg$ may be viewed as a detour
towards classification which is forced by the lack of Lie groups which 
integrate the Lie algebras $\algebroidq$ and $\algebroidq$. 
Integrability of Lie algebroids is an interesting question, and the 
reader is referred to \cite{Integrability-Lie-Algebroid} for some nice 
results/obstructions.
\end{remark}

\begin{thm}\label{thm:gauge-group}
The groups $\Omega^0(M,\Groupg)$  
and $\Omega^0(M,\Groupq)$  act on 
 $\CM(M,\algebroidg)$ and $\CM(M,\algebroidq)$, 
respectively.
Over the space $\Fcal(M)$ of smooth transversely oriented  
codimension-one foliations of $M$, the actions of 
$\Omega^0(M,\Groupg)$ and $\Omega^0(M,\Groupq)$ 
preserve the fibers of 
\[\proj_{\mathsf{g}}:\CM(M,{\algebroidg^*})\ra \Fcal(M)
\ \ \ \text{and}
\ \ \ \proj_{\mathsf{q}}:\CM(M,{\algebroidq^*})\ra \Fcal(M).\]
while the actions are transitive and without  
fixed  points on the fibers. 
\end{thm}
\begin{proof}
Let us assume that $\oY\in\Omega^0(M,\Groupg)$ and 
$\quot(\oY)=\sum_{m=1}^\infty y_i\h^i$.
The foliations given by $\oA$ and $\oY\star\oA$ are defined 
by the $1$-forms $\oA(0)=a_0$ and $\oA(0)/\oY'(0)=a_0/y_1$, 
respectively. Since $y_1$ is a positive valued function on $M$, 
it follows that the action of $\Omega^0(M,\Groupg)$
preserves the fibers of $\proj_{\mathsf{g}}$. 
We then need to show that this latter action is transitive and 
without fixed points. Fix $\Fol\in\Fcal(M)$ and 
the transverse vector field $V$. Let $\oA=\oA_{\Fol,V,1}$ and 
given a section $\oX\in\Omega^0(M,\algebroidg)$ with 
$\oX(0)<0$, solve the equation
\begin{align*}
\oX=\imath_V(\ob\star\oA)
=\imath_V\left(\frac{\oA\circ\ob-d\ob}{\ob'}\right)
=-\frac{1+L_V(\ob)}{\ob'}
\end{align*}
for $\ob\in \Omega^0(M,\Groupg)$. Evaluation at $0$ gives 
$x_0+1/y_1=0$, which implies $y_1=-1/x_0>0$. Furthermore, the 
above differential equation uniquely determines $\ob$ in a 
neighborhood of the origin in $\R$ in terms of the given  
$\oa\in\Omega^0(M,\algebroidg)$ and $\oA$. This completes 
the proof of the theorem for germ cords with the help of 
Proposition~\ref{prop:MC-elements}.\\

It follows that the gauge action of $\Groupq$ preserves the 
fibers of $\proj_{\mathsf{q}}$ and that the action is transitive
over the fibers. If 
$\qY=\sum_{m=1}^\infty y_m\h^m\in\Omega^0(M,\Groupq)$ preserves 
a quantum cord 
\[\qA=\sum_{m=0}^\infty a_m\h^m\in\CM(M,\algebroidq)\]
with $a_0\neq 0$, we find 
$\qY'\qA=\qA\circ\qY-d\qY$. We may contract this equation using 
a vector field $V$ which is transverse to the foliation induced by 
$a_0$ and satisfies $\imath_V(a_0)=1$, to get
\[\qY'\imath_V(\qA)=\imath_V(\qA)\circ\qY-L_V\qY\]
Let us assume that $\imath_V(\qA)=\sum_{m=0}^\infty b_m\h^m$, 
where $b_0=1$. The initial term in the above equation 
reads as $y_1=1$. Looking at the coefficient of $\h^n$ 
gives an equation of the from
\[y_{n+1}=F_n(y_1,\ldots,y_n,b_1,\ldots,b_n)\]
which uniquely determines $y_{n+1}$ by induction. Since $y_1=1$ and 
$y_m=0$ for $m>1$ is an obvious solution, it follows that this is 
the only possibility. This completes the proof of the theorem.
\end{proof}

\section{Leaves and holonomy}
\label{sec:holonomy}
\subsection{Impotent cords}
Let us fix a germ cord $\oA\in\Cordobag$. If $\oA(0)$ is a nowhere 
zero $1$-form, it defines a transversely oriented codimension one 
foliation $\Fol=\proj(\oA)\in\Fcal(M)$. 
Associated with every leaf $L$ 
of $\Fol$ we obtain an immersion $\imath_L:L\ra M$ which gives the 
restriction $\oA|_{L}=\imath_L^*\oA\in\CM(L,\algebroidg)$.
The nature of this germ cord is quite different from the nature 
of $\oA$ in the following sense. Unlike $\oA(0)$ which is nowhere 
zero, $\oA|_{L}$ takes its values in $\algebrag\subset \algebroidg$ 
and $\oA|_{L}(0)=0$.

\begin{defn}\label{defn:impotent}
A germ cord $\oA\in\Cordobag$ is called {\emph{impotent}} if 
$\oA\in\Omega^1(M,\algebrag)$. Similarly, $\qA\in\Cordobaq$ is
called {\emph{impotent}} if $\qA\in\Omega^1(M,\algebraq)$. 
The spaces of impotent germ cords and impotent quantum cords 
are denoted by $\Cordobaig$ and $\Cordobaiq$, respectively.
\end{defn}

\begin{lem}[Poincar\'e Lemma]\label{lem:locally-trivial}
Every impotent germ cord $\oA\in\Cordobaig$ is locally gauge
equivalent to zero, i.e. every $x\in M$ has an open neighborhood
 $U_x\subset M$  so that $\oA|_{U_x}$ is 
gauge equivalent to zero. Similarly, every impotent quantum cord 
$\qA\in\Cordobaiq$ is locally gauge equivalent to zero.
\end{lem}
\begin{proof}
Choose coordinates $(x_1,\ldots,x_n)$ on a chart $U$ 
around $x$ so that $x$ corresponds to the origin and 
$\oA$ is given by $\sum_if_idx_i$ with 
$f_i\in \Omega^0(U,\algebrag)$.
After shrinking $U$, we can assume that for some 
$\epsilon>0$, the function $f_i$ is defined for 
all $(\h,x)\in W=(-\epsilon,\epsilon)\times U$.
From $d\oA=\oA'\oA$ we get
\[\partial_if_j+f_i\partial_\h f_j=\partial_jf_i+f_j\partial_\h f_i
\ \ \ \ \forall\ \ i,j\in\{1,\ldots n\}.\]
The above equation implies that the vector fields 
$\xi_i=\partial_i+f_i\partial_\h$ commute. We can thus choose new
coordinates $(y_0,y_1,\ldots,y_n)$ on an open neighborhood $W'$ 
of $(0,x)\in W$ so that $\partial/\partial y_i=\xi_i$, 
$y_0$ agrees with $\h$ over $x$ and the foliation is 
given by $\{y_0=\text{constant}\}$. 
Choose $U_x\subset U$ such that it contains $x$ and
$(-\delta,\delta)\times U_x$ is a subset of $W'$. Set $\oY$ equal to 
$y_0$ over $U_x$. It is then clear that 
$\oY\in \Omega^0(U_x,\Groupg)$. Furthermore, we have 
$\xi_i\oY=0$, which means that $\partial_i\oY+f_i\oY'=0$.
This means that $\oA|_{U_x}=-d\oY/\oY'=\oY\star 0$
and completes the proof for impotent germ cords. The statement 
for impotent quantum cords follows from a similar argument.
\end{proof}

The gauge group sends impotent cords to impotent cords. After 
dividing  by the action of the gauge group, we obtain the 
moduli spaces 
\begin{align*}
\Moduliig&=\Cordobaig/\Gaugeg\ \ \ \text{and}\\
\Moduliiq&=\Cordobaiq/\Gaugeq,
\end{align*}
called the moduli spaces of impotent germ cords and 
impotent quantum cords, respectively.

\begin{prop}\label{prop:classification-impotent}
For every smooth manifold $M$, there are
natural bijections
\begin{align*}
&\rho_{M}^\sg:\Moduliig\lra 
\ov\Hom(\pi_1(M),\Groupg):=\Hom(\pi_1(M),\Groupg)/\Groupg
\ \ \ \ \text{and}\\
&\rho_{M}^\sq:\Moduliiq\lra 
\ov\Hom(\pi_1(M),\Groupq):=\Hom(\pi_1(M),\Groupq)/\Groupq.
\end{align*}
If $f:M_1\ra M_2$ is a smooth map between smooth 
manifold, the following diagram is commutative:
\begin{displaymath}
\begin{diagram}
\Moduliign2&&\rTo{f^*}&&\Moduliign1&&\\
&\rdTo{\quot}&&
&\vLine{\rho_{M_1}^\sg}&\rdTo{\quot}&\\
\dTo{\rho_{M_2}^\sg}
&&\Moduliiqn2&\rTo{f^*\ }&  \HonV  &&\Moduliiqn1\\
&&\dTo_{\rho_{M_2}^\sq}&&\dTo{}&&\dTo{\rho_{M_1}^\sq}\\
\ov\Hom(\pi_1(M_2),\Groupg)&\hLine&\VonH&\rTo{f^*\  }&
\ov\Hom(\pi_1(M_1),\Groupg)&&\\
&\rdTo{\quot}&&&&\rdTo{\quot}&\\
&&\ov\Hom(\pi_1(M_2),\Groupq)&&\rTo{f^*\ }  
 &&\ov\Hom(\pi_1(M_1),\Groupq)\\
\end{diagram}
\end{displaymath}
\end{prop}   
\begin{proof}
The proof is identical with the proof of 
Theorem~\ref{thm:classification-cords}
for the most part, as is sketched below. Pick 
$\oA\in\Cordobaig$ and cover $M$ with open subsets 
$U_\alpha$ so that $\oA|_{U_\alpha}=\oY_\alpha\star 0$ for
$\oY_\alpha\in\Omega^0(U_\alpha,\Groupg)$. Over 
$U_\alpha\cap U_\beta$, the transition functions 
$c_{\alpha\beta}=\oY_\alpha\circ\oY_\beta^{-1}
:U_\alpha\cap U_\beta\ra \Groupg$ are then locally constant,
since $c_{\alpha\beta}\star 0=0$.
These maps  define a cohomology class
in the \v{C}ech cohomology  $\HC^1(M,\Groupg)$.
The functions $\oY_\alpha$ are well-defined only upto composition
with locally constant functions, but this freedom does not 
change the cohomology class, as before.
Moreover, if we gauge $\oA$ by $\oY\in\Omega^0(M,\Groupg)$, 
the  cocycles associated with $\oA$ and $\oY\star \oA$ are the same 
and we obtain the map 
\begin{align*}
\rho=\rho_{M}^\sg:\Moduliig\ra 
\HC^1(M,\Groupg)\simeq \Hom(\pi_1(M),\Groupg)/\Groupg.
\end{align*}
The argument of Theorem~\ref{thm:classification} may be copied
to show that $\rho_{M}^\sg$ is injective.
This implies that over simply connected domains, every 
impotent germ cord is gauge equivalent to zero. 
\\

If $\{c_{\alpha\beta}:U_\alpha\cap U_\beta\ra \Groupg\}$ 
is a cocycle in $\HC^1(M,\Groupg)$, we can choose
 a smooth partition of unity 
 $\{\lambda_\alpha:U_\alpha\ra \R^{\geq 0}\}_\alpha$ as before
and  define $\oY_\alpha:U_\alpha\ra \Groupg$ by 
$\oY_\alpha(\h,x)=
\sum_{\gamma} \lambda_\gamma(x) c_{\gamma\alpha}(\h)$.
This is well-defined as a germ and we have 
$\oY_\alpha,\oZ_\alpha\in\Omega^0(U_\alpha,\Groupg)$ 
where $\oZ(\cdot,x)=\oY(\cdot,x)^{-1}$.
The germs $\oZ_\alpha\star 0\in\CM(U_\alpha,\algebrag)$ 
match over the intersections $U_\alpha\cap U_\beta$.
They can thus be glued to give some $\oA\in\Cordobaig$
with $\rho_{M}^\sg(\oA)=[\{c_{\alpha\beta}\}]\in 
\HC^1(M,\Group)$. The proof for impotent quantum cords is 
completely similar. The commutativity of the cubic diagram 
is straight-forward from the definitions. 
\end{proof} 

\subsection{Monodromy for impotent cords}
Let us assume that $\oA\in\Cordobaig$ is an impotent germ cord.
Every element $\oY\in\Groupg$ defines the map
\[D\ob:\algebrag=T_{Id}\Groupg\ra T_{\ob}\Groupg.\]
We can use this map to define  a {\emph{connection}} 
$H^\oA\subset TM\times T\Groupg$ by 
\begin{align*}
H_{x,\ob}^\oA=\left\{(\zeta,D\ob(\oA(\zeta)))\ \big|\ 
\zeta\in T_x M\right\}.
\end{align*} 
Since $\oA$ satisfies $d\oA+[\oA,\oA]/2=0$, it follows that 
that $H^\oA$ gives a foliation $\Fol^\oA$ of $M\times \Groupg$ 
and a foliation $\wt\Fol^\oA$ of $\wt{M}\times \Groupg$, 
where $\wt{M}$ denotes the universal cover of $M$. For  
constructing this foliation, the weak 
notions of smoothness on $\algebroidg$ and $\Groupg$ suffice.\\

Fix a point $x\in M$ and a corresponding pre-image 
$\wt{x}\in\wt{M}$ of $x$ under the covering map. Every 
$\theta\in\pi_1(M,x)$ may be lifted to a path $\wt{\theta}$ 
on the leaf of $\wt\Fol^\oA$ which passes through 
$(\wt{x}, Id_{\Groupg})\in\wt{M}\times \Groupg$. The 
{\emph{monodromy map}}
\[\phi=\phi_\oA:\pi_1(M,x)\ra \Groupg,\ \ \ \ \ 
\phi(\theta):=\pi_{\Groupg}(\wt{\theta}(1))\]
is defined by projecting 
$\wt{\theta}(1)\in\wt{M}\times\Groupg$ onto its second factor. 
It follows that 
$\wt{\theta}(1)=(\theta\wt{x},\phi_M(\theta))$, where 
$\theta\wt{x}$ denotes the image of $\wt{x}$ under the deck 
transformation corresponding to $\theta$. Moreover, since 
$(\wt{x},Id_{\Groupg})$ and $(\theta\wt{x},\phi(\theta)$
are on the same leaf of $\wt{\Fol}^\oA$, for every 
$\ob\in\Groupg$ the points $(\wt{x},\ob)$ and 
$(\theta\wt{x},\phi(\theta)\circ \ob)$ are also on the same 
leaf of $\wt{\Fol}^\oA$.\\

Every other pre-image of $x$ under the covering map is of the form
$\wt{y}=\gamma\wt{x}$. If we use $\wt{y}$ instead of $\wt{x}$ 
in defining $\phi$, we obtain another map $\phi'$, with the 
property that the points $(\wt{y},Id_{\Groupg})$ and 
$(\theta\wt{y},\phi'(\theta))
=(\theta\gamma\wt{x},\phi'(\theta))$ are on the same leaf.
On the other hand $(\wt{y},Id_{\Groupg})$ is on the same leaf as 
$(\wt{x},\phi(\gamma)^{-1})$. If follows that
\[\phi'(\theta)=\phi(\gamma)^{-1}\phi(\theta)\phi(\gamma),
\ \ \ \ \forall\ \theta\in\pi_1(M,x).\]
In particular, the conjugacy class of the representation 
$\phi_\oA:\pi_1(M,x)\ra \Groupg$ does not depend on the choice 
of the pre-image $\wt{x}$ of $x$.
On the other hand, if we gauge the germ cord $\oA$ by a section 
$\ob\in\Omega^0(M,\Groupg)$, one can easily show that the 
monodromy map $\phi:\pi_1(M)\ra \Groupg$ changes by 
conjugation by $\ob(x)\in\Groupg$.\\

The above discussion gives a second construction which 
constructs the map
\[\rho_{M}^\sg:\Moduliig\ra \ov\Hom(\pi_1(M),\Groupg)=
\Hom(\pi_1(M),\Groupg)/\Groupg\]
in an explicit way, by assigning the monodromy homomorphism
$\phi_\oA\in\ov\Hom(\pi_1(M),\Groupg)$ to every 
$\oA\in\Moduliig$.
A similar discussion gives an explicit description 
of the correspondence
\[\rho_{M}^\sq:\Moduliiq\ra \ov\Hom(\pi_1(M),\Groupq)=
\Hom(\pi_1(M),\Groupq)/\Groupq\]
by assigning the {\emph{quantum monodromy}} map
$\phi_\qA$ to $\qA\in\Moduliiq$.  \\

There is a third (geometric) way to understand the monodromy map 
as follows. Let us assume that $\oA\in\Cordobaig$ is represented by 
a smooth differential form on $(-\epsilon,\epsilon)\times M$ such 
that $\oA(0,x)=0$ for all $x\in M$. As discussed in the proof of 
Lemma~\ref{lem:locally-trivial}, $\oA$ defines a foliation on 
$(-\epsilon,\epsilon)\times M$. In fact, the $1$-form 
$\oB=\oA-d\h\in\Omega^1((-\epsilon,\epsilon)\times M,\R)$ satisfies
$d\oB=d\oA-\oA'd\h=\oA'\oB$, which implies the Frobenius condition
$\oB d\oB=0$. It thus gives a foliation $\Fol_\oA$ on 
$(-\epsilon,\epsilon)\times M$. Since $\oA(0,x)=0$, $\{0\}\times M$
is one of the leaves of $\Fol_\oA$. Let us fix $x\in M$ and 
$\gamma:[0,1]\ra M$ so that $\gamma(0)=\gamma(1)=x$. The positive 
number $\delta>0$ may be chosen so that $\gamma$ may be lifted
(in a unique way) to a curve 
$\gamma_\h:[0,1]\ra (-\epsilon,\epsilon)\times M$
with image on the leaf passing through $(\h,x)$ so that 
$\gamma_\h(0)=(\h,x)$ and $\pi_M(\gamma_\h(s))=\gamma(s)$ 
for all $s\in[0,1]$. Here $\pi_M:(-\epsilon,\epsilon)\times M\ra M$ 
denotes the projection map over the second factor, while 
the projection map over the first factor is denoted by $\pi_\R$.
It is easy to show that the value 
$\pi_\R(\gamma_\h(1))\in\R$ is independent of the choice of 
$\gamma$ in its homotopy class $[\gamma]\in\pi_1(M,x)$.
Let us denote this value by $\phi_{[\gamma]}(\h)$. 
Since $\{0\}\times M$  is a leaf, $\phi_{[\gamma]}(0)=0$. 
It follows that $\phi_{[\gamma]}$ is smooth and that the map 
$\phi:\pi_1(M,x)\ra\Groupg$ which sends $[\gamma]$ to the 
germ of $\phi_{[\gamma]}$ is a homomorphism. Moreover,
the conjugacy class of this homomorphism remains invariant 
under gauge, and is equal to $\phi_\oA$. This point of view 
brings us very close to the notion of holonomy for the leaves of 
a foliation on $M$.\\

\subsection{Holonomy of leaves}
Let us assume that $\Fol\in\Fcal(M)$ is a transversely oriented 
codimension one foliation on $M$ and that $L$ is a leaf of $\Fol$.
$\Fol$ corresponds to the gauge equivalence class of a 
germ cord $\oA\in\Cordobag$. The restriction $\oA|_{L}$ of 
$\oA$ to $L$ is impotent and we thus obtain a homomorphism
$\phi_{\Fol,L}\in\ov\Hom(\pi_1(L),\Groupg)$. If $x\in L$ is a 
fixed point, using a transverse arc we can also define a 
{\emph{holonomy}} homomorphism $\rho_L:\pi_1(L,x)\ra \Groupg$,
and the conjugacy class of this homomorphism is independent of 
the choice of $x$ and the transverse arc.\\

\begin{prop}\label{prop:holonomy-of-leaves}
For every leaf $L$ of a smooth transversely oriented 
codimension-one foliation $\Fol$ of a smooth manifold $M$, the 
conjugacy classes of the holonomy homomorphism 
$\rho_L:\pi_1(L,x)\ra \Groupg$ and the monodromy 
representation $\phi_{\Fol,L}:\pi_1(L,x)\ra\Groupq$ in 
$\ov\Hom(\pi_1(L),\Groupg)$ are the same.
\end{prop}
\begin{proof}
Let us assume that $\Fol\in\Fcal(M)$ is a transversely oriented 
codimension one foliation on $M$ given by a $1$-form 
$a\in\Omega^1(M,\R)$, $V$ is a transverse vector field with 
$\imath_V(a)=-1$ and $\oA=\oA_{a,V}$ is the corresponding cord.
Denote the flow of $V$ by $\Phi_\h$ (thus, $\oA=\Phi_\h^*(a)$).
Associated with $\oA$ we obtain a $1$-form 
$\oB=\oA-d\h=\Phi_\h^*(a)-d\h\in\Omega^1(\R\times M,\R)$ 
and a foliation $\Fol_\oA$ on $\R\times M$ as before.  
If we define $F:\R\times M\ra \R\times M$ by 
$F(\h,x)=(\h,\Phi_\h(x))$, it follows that $\oB=F^*(a)$. 
The foliation $\Fol_\oA$ is thus given as the image of the product 
foliation $\R\times \Fol$ on $\R\times M$ under the map $F$.\\

Suppose that $L$ is a leaf of $\Fol$ and fix  $x\in L$.
Our third description of the monodromy map 
\[\rho_{L}^\sg:\Mod(L,\algebroidg,\Groupg)\ra 
\ov\Hom(\pi_1(L),\Groupg)\]
may be used to describe the homomorphism $\phi=\phi_{\oA|_{L}}$ as 
follows. The foliation associated with $\oA|_L$ 
is the restriction of 
$\Fol_\oA$ to $(-\epsilon,\epsilon)\times L\subset \R\times M$.
For every small value of $\h$, the curve $\gamma_\h$ is mapped
to a curve $\theta_\h=F\circ \gamma_\h$ by $F$. Note, however, that
$\pi_\R\circ \theta_\h=\pi_\R\circ \gamma_\h$. In particular, 
$\phi=\phi_{[\gamma]}\in\Groupg$ may be computed as the return map 
of the curves $\{\theta_\h\}_\h$ for small values of $\h$.
We then observe that $\theta_\h(0)=F(\h,x)=(\h,\Phi_\h(x))$.
Moreover, $\theta_\h(1)=F(\phi(\h),x)=(\phi(\h),\Phi_{\phi(\h)}(x)$.
We can parametrize the transverse arc 
$\{\Phi_\h(x)\ |\ \h\in (-\epsilon,\epsilon)\}$ 
to the foliation $\Fol$ in $M$
by $\h\in  (-\epsilon,\epsilon)$ and then the above considerations
imply that $\rho_L(\gamma)(\h)=\phi(\h)$, completing the proof
of the proposition.
\end{proof}

Suppose that a foliation $\Fol\in\Fcal(M)$ is compatible with 
 a germ cord $\oA\in\Cordobag$ and that 
 $\quot(\oA)=\sum_{m}a_m\h^m$.
Proposition~\ref{prop:holonomy-of-leaves} and 
Proposition~\ref{prop:classification-impotent}
imply that in the Taylor expansion of the holonomy map 
along a closed curve $\gamma$, the initial term is obtained by 
integrating $a_1$ along $\gamma$. This observation generalizes a 
proposition of Ghys in \cite{Ghys-GV-1}, which identifies the 
first derivative of the holonomy map for a foliation $\Fol$ given 
by a $1$-form $a_0\in\Omega^1(M,\R)$ with the integral of 
$a_1=L_V(a_0)$ along the closed curves representing the elements 
of $\pi_1(L,x)$. \\

The above observations suggest the following extension of the 
concept of leaves and their holonomy to singular foliations.

\begin{defn}
Let $\Fol\in\Modulig$ denote the gauge equivalence class of 
$\oA\in\Cordobag$. A {\emph{leaf-like}} map for 
the singular foliation $\Fol$ is a diffeomorphism $f:L\ra M$ from 
a smooth manifold $L$ to $M$ such that $f^*\oA$ is impotent. 
\end{defn}

The definition is clearly independent of the choice of the 
representative $\oA$ for the singular foliation $\Fol$. Associated 
with a leaf-like map $f:L\ra M$, we obtain the conjugacy class 
of a holonomy map
\[\rho_{\Fol,L}=\rho_{L}^\sg(\oA|_{L})
\in\ov\Hom(\pi_1(L),\Groupg)\simeq 
\Mod(L,\algebrag,\Groupg).\]
This notion of holonomy generalizes the usual holonomy map for 
the leaves of non-singular foliations, by 
Proposition~\ref{prop:holonomy-of-leaves}.
This generalization may be compared with other generalizations of 
the notion of holonomy for singular foliations, and in particular 
\cite{Holonomy-Groupoid}.

\section{Classification of germ cords and quantum cords}
\label{sec:classification}
\subsection{A Poincar\'e lemma for cords}
The purpose of this section is to state and prove a classification 
theorem for germ cords and quantum cords up to gauge equivalence.
A survey of approached to classification of foliations may be found
in \cite{Hurder-survey}.
The basis of any such theorem is a local classification lemma, which 
shows that up to gauge equivalence, every cord is locally trivial.
We refer to such  statement as a {\emph{Poincar'e Lemma}}.
 
\begin{lem}[Poincar\'e Lemma]\label{lem:Poincare-Lemma}
Every germ cord $\oA\in \CM(M,\Algebroidg)$ is locally 
gauge equivalent to the trivial cord, i.e. for every $x\in M$ 
there is an open neighborhood 
$U_x\subset M$ of $x$ such that $\oA|_{U_x}=\oY_x\star 0$ for 
some $\oY_x\in\Omega^0(U_x,\Groupoidg)$. 
\end{lem}
\begin{proof}
Let us assume that $\oA\in  \CM(M,\Algebroidg)$ 
is a germ cord and that
$s_\oA:M\ra \R$ is the corresponding source map. 
One can then represent $\oA$ as the germ of a differential form
$\oA\in\Omega^1(U,\R)$, where $U$ is an open neighborhood of 
\[\Delta_\oA=\{(s_\oA(x),x)\in\R\times M\ |\ x\in M\}.\]
By making $U$ smaller, if necessary, we can assume that 
$\oA$ satisfies $d_M\oA=\oA'\oA$, which is equivalent to 
$d_U(\oA-d\h)=\oA'(\oA-d\h)$. In particular, 
$\oB=\oA-d\h\in \Omega^1(U,\R)$ defines a codimension one
foliation $\Fol^\oA$ on $U$ which is transverse to vertical 
lines $\ell_y=U\cap (\R\times \{y\})$ for all $y\in M$.
For every $u\in U$ let us denote the leaf of $\Fol^\oA$ through 
$u$ by $L_u$. Given $x\in M$ we can choose an open neighborhood 
$U_x\subset M$ of $x$ and an open subset 
$\ell'_x\subset \ell_x\subset\R\times \{x\}$ 
which contains $(s_\oA(x),x)\in\Delta_\oA$,
such that the union of leaves of the foliation $\Fol^\oA$ which 
cut $\ell'_x$, intersect $U\cap (\R\times U_x)$ in a {\emph{box}}
$W$ around $x$. By this, we mean that associated with every 
$y\in U_x$ and every $r\in \ell'_x$ there is a unique 
point $w=w(r,y)\in \ell_y\cap W$ such that the
connected component of $L_w\cap W$ which contains $w$ also
contains $r$. Moreover, every $w\in W$ is of the form $w(r,y)$ 
for some $y\in U_x$ and some $r\in \ell'_x$.
Over the box $W$, we can define 
the real-valued function $\oY_x$ so that the restriction of  
$\oY_x$ to every {\emph{plaque}}
\[P_r=\{w(r,y)\ |\ y\in U_x\}\subset W\ \ \ 
\text{for}\ r\in \ell'_x\]
 is constant. The map $\oY_x$ defines a smooth function 
 $\oY_x:U_x\ra \Groupoidg$ such that
$\oA|_{U_x}=\oY_x\star 0=-d\oY_x/\oY_x'$. 
In particular, every germ cord 
$\oA\in\Cordobag$ is {{locally gauge equivalent to zero}}.
\end{proof}
 
\begin{remark}
Our earlier assumption that every quantum cord is the 
locally trivial is made to replace the above lemma, which is 
only available for germ cords. It follows that the image of 
every germ cord under $\quot$ is automatically locally trivial,
and is thus a quantum cord.
\end{remark}

Considering the full action of the gauge 
groupoids on spaces of cords gives the moduli spaces
\[\Mod(M,\Algebroidg,\Groupoidg)=
\CM(M,\Algebroidg)/\Omega^0(M,\Groupoidg)\ \ \ \text{and}
\ \ \ \Mod(M,\Algebroidq,\Groupoidq)=
\CM(M,\Algebroidq)/\Omega^0(M,\Groupoidq),\]
which are called the {\emph{moduli space of germs cords}} and the 
{\emph{moduli space of quantum cords}}, respectively.
Theorem~\ref{thm:gauge-group} shows that the space 
$\Fcal(M)$ of smooth transversely oriented codimension one 
foliations on $M$ may be identified with a subset of both of 
the moduli spaces. In fact,  Remark~\ref{rmk:group-to-groupoid}
implies that there are  bijections
\[\Ibb^\sg:\Mod(M,\algebroidg,\Groupg)\ra 
\Mod(M,\Algebroidg,\Groupoidg)\ \ \ \text{and}\ \ \ 
\Ibb^\sq:\Mod(M,\algebroidq,\Groupq)\ra 
\Mod(M,\Algebroidq,\Groupoidq),
\]  
which sit in a commutative diagram
\begin{equation}\label{eq:T-is-surjective}
\begin{diagram}
\Fcal(M)&\rInto{\ \ \ \ \ \ } &\Mod(M,\algebroidg,\Groupg)&
\rTo{\ \ \Ibb^\sg\ \ }&
\Mod(M,\Algebroidg,\Groupoidg)\\
\dTo{Id}&&\dTo{\quot}&&\dTo{\quot}\\
\Fcal(M)&\rInto{\ \ \ \ \ \ } &\Mod(M,\algebroidq,\Groupq)&
\rTo{\ \ \Ibb^\sq\ \ }&
\Mod(M,\Algebroidq,\Groupoidq).
\end{diagram}
\end{equation}

\subsection{The \v{C}ech cohomology}
Recall that a map $\oY:M\ra \Groupoidg$, which assigns an arrow 
$\oY(x)$ in $\Groupoidg$ to each point $x\in M$, is smooth if for 
every $x\in M$ one can find an open set $U$ containing $x$ 
so that $\oY|_{U}$ can be represented by a smooth real-valued 
map on $\R\times U$ which is still denoted by $\oY$, so that 
$\oY(x)$ is given by $\oY(\cdot,x):\R\ra\R$ at each point 
$x\in U$. 
The smooth section $\oY:M\ra \Groupoidg$ is 
{\emph{locally constant}} if for every point $x\in M$ 
there is a smooth 
local diffeomorphism  $f$ from a neighborhood of $s_\oY(x)$ 
to a neighborhood of $t_\oY(x)$ and a neighborhood $U$ of $x$ 
so that $\oY(y)$ is given as the germ of $f$ at $s_\oY(y)$ for 
every point $y\in U$. Similarly,
a map $\qY:M\ra \Groupoidq$ from $M$ to the arrows of the 
groupoid $\Groupoidq$ which is given by  
$\qY=\sum_{m=0}^\infty y_m(x)(\h-s(x))^m$
is  {{smooth}} if  the functions $y_m:M\ra \R$ and $s=s_\qY:M\ra \R$ 
are smooth, and is called {\emph{locally constant}} if
$dy_m=(m+1)y_{m+1}ds$ for  all $m\in\Z^{\geq 0}$.
This condition may be described as 
\[d\qY=\sum_{m=0}^\infty \left(dy_m-(m+1)y_{m+1}ds\right)
(\h-s)^m=0.\]
Locally constant maps to $\Groupoidg$ and $\Groupoidq$ are in fact 
smooth maps to $\Gamma^\sq$ and $\Gamma^\sq$, respectively.\\

The spaces of locally constant functions with values in 
$\Gamma^\sg$ and $\Gamma^\sg$ over a manifold $M$
is denoted by $\Omega^0(M,\Gamma^\sg)$ and 
$\Omega^0(M,\Gamma^\sq)$, respectively.
Correspondingly, we can define  the \v{C}ech cohomology groups 
$\HC^1(M,\Gamma^\sg)$ and $\HC^1(M,\Gamma^\sq)$. For this 
purpose, associated with each open cover 
$\Ucal=\{U_\alpha\}_{\alpha}$ of $M$, we can construct 
the spaces of cocycles $\Ccal^1(\Ucal,\Gamma^\sg$ and 
$\Ccal^1(\Ucal,\Gamma^\sq)$, as well as the spaces of coboundaries
$\Bcal^1(\Ucal,\Gamma^\sg)$ and $\Bcal^1(\Ucal,\Gamma^\sq)$.
An element of $\Ccal^1(\Ucal,\Groupoidg)$ consists of a union of 
locally constant maps 
$\cg_{\alpha\beta}:U_\alpha\cap U_\beta\ra \Gamma^\sg$
from $U_\alpha\cap U_\beta$ to the arrows of $\Gamma^\sg$
which satisfy the cocycle condition 
$\cg_{\alpha\beta}\circ\cg_{\beta\gamma}=\cg_{\alpha\gamma}$.
In other words, a cocycle in $\Ccal^1(\Ucal,\Gamma^\sg)$ is a 
continuous groupoid homomorphism from $\Gamma^\Ucal$ 
to $\Gamma^\sg$. The space $\Ccal^1(\Ucal,\Gamma^\sq)$ is defined 
similarly using locally constant maps with values in $\Groupoidq$
and a cocycle in $\Ccal^1(\Ucal,\Gamma^\sq)$ is a continuous 
groupoid homomorphism from $\Gamma^\Ucal$ to $\Gamma^\sq$.
 The space of 
coboundaries $\Bcal^1(\Ucal,\Gamma^\sg)$ consists of a union 
of locally constant maps 
\[\cg_{\alpha\beta}:U_\alpha\cap U_\beta\ra \Gamma^\sg\]
which come from  locally constant maps 
$\{\bg_{\alpha}:U_\alpha\ra \Gamma^\sg\}_\alpha$ 
in the sense that $\cg_{\alpha\beta}=\bg_\alpha\circ \bg_\beta^{-1}$
over the intersections $U_\alpha\cap U_\beta$. 
The coboundaries are the groupoid homomorphisms from 
$\Gamma^\Ucal$ to $\Gamma^\sg$ which are conjugate to 
the trivial homomorphism. Again, we can define 
$\Bcal^1(\Ucal,\Gamma^\sq)$ in a similar way. We then set
\[\HC^1(\Ucal,\Gamma^\sg):=\Ccal^1(\Ucal,\Gamma^\sg)/
\Bcal^1(\Ucal,\Gamma^\sg)\ \ \ \text{and}\ \ \ 
\HC^1(\Ucal,\Gamma^\sq):=\Ccal^1(\Ucal,\Gamma^\sq)/
\Bcal^1(\Ucal,\Gamma^\sq).\]
Considering the refinements of the coverings, we can define 
the limits of $\HC^1(\Ucal,\Gamma^\sg)$ and 
$\HC^1(\Ucal,\Gamma^\sq)$, which are $\HC^1(M,\Gamma^\sg)$ and 
$\HC^1(M,\Gamma^\sq)$, respectively.
The quantization  functor 
$\quot:\Gamma^\sg\ra\Gamma^\sq$  induces the maps 
\[\quot:\Ccal^1(\Ucal,\Gamma^\sg)\ra \Ccal^1(\Ucal,\Gamma^\sq),\ \ \
\quot:\Bcal^1(\Ucal,\Gamma^\sg)\ra \Bcal^1(\Ucal,\Gamma^\sq)\ \ \
\text{and}\ \ \ \quot:\HC^1(M,\Gamma^\sg)\ra \HC^1(M,\Gamma^\sq).\]
It can be shown that 
$\HC^1(M,\Gamma^\sg)\simeq\HC^1(\Ucal,\Gamma^\sg)$ and 
$\HC^1(M,\Gamma^\sg)\simeq \HC^1(\Ucal,\Gamma^\sg)$
if the cover $\Ucal=\{U_\alpha\}_\alpha$ consists only
of contractible open subsets of $M$.\\

\begin{thm}\label{thm:classification}
There are natural one to one correspondences
\begin{align*}
\cg:\Mod(M,\Algebroidg,\Groupoidg)\ra \HC^1(M,\Gamma^\sg)
\ \ \ \text{and}\ \ \ 
\cq:\Mod(M,\Algebroidq,\Groupoidq)\ra\HC^1(M,\Gamma^\sq).
\end{align*}
from the moduli space of germ cords and yje moduli space of quantum 
cords to the \v{C}ech cohomology spaces with 
coefficients in $\Gamma^\sg$ and $\Gamma^\sq$, respectively.
\end{thm}
\begin{proof}
Given $\oA\in\CM(M,\Algebroidg)$, we can cover $M$ 
with finitely many open 
sets $\{U_\alpha\}_{\alpha}$ so that $\oA|_{U_\alpha}$ is 
gauge equivalent to zero. One can then pick the sections
$\oY_\alpha\in\Omega^0(U_\alpha,\Groupoidg)$ such that 
$\oA|_{U_\alpha}=\oY_\alpha\star 0$. Over the intersections 
$U_\alpha\cap U_\beta$ we obtain 
$\oY_\alpha\star0=\oY_\beta\star 0$, which implies that 
$d(\oY_\alpha\circ\oY_\beta^{-1})=0$, or that 
\[\cg_{\alpha\beta}=\oY_\alpha\circ\oY_\beta^{-1}:
U_\alpha\cap U_\beta\ra \Groupoidg\] is locally 
constant. Note that over $U_\alpha\cap U_\beta$, we have
$s_{\oY_\alpha}=t_{\oY_{\beta}^{-1}}=s_{\oY_\beta}=s_\oA$ and 
the compositions $\cg_{\alpha\beta}=\oY_\alpha\circ\oY_\beta^{-1}$
are thus well-defined.
The above argument gives the smooth locally constant maps
\[\cg_{\alpha\beta}=\oY_\alpha\circ\oY_\beta^{-1}:
U_\alpha\cap U_\beta\ra \Gamma^\sg\ \ \ 
\leadsto\ \cg(\oA)=\left[\{\cg_{\alpha\beta}\}_{\alpha,\beta}\right]
\in \HC^1(M,\Gamma^\sg).\]
 It is not hard to see that this class is well-defined. The role of 
 the covering is not important as we can always
pass to a common refinement for two different given covers.
If $\oA|_{U_\alpha}=\oX_\alpha\star 0=\oY_\alpha\star 0$, then 
$d_\alpha=\oX_\alpha\circ\oY_\alpha^{-1}:U_\alpha\ra \Gamma^\sg$
is locally constant and 
\[d_{\alpha}\circ(\oY_\alpha\circ\oY_\beta^{-1})=
\oX_\alpha\circ\oY_\beta^{-1}=
(\oX_\alpha\circ\oX_\beta^{-1})\circ d_\beta
\ \ \ \ \forall\ \alpha,\beta\]
which implies that the definition of $\cg(\oA)$ is independent of 
the choice of $\{\oY_\alpha\}_\alpha$. The map
\[\cg:\CM(M,\Algebroidg)\ra \HC^1(M,\Gamma^\sg)\]
is thus well-defined.\\

We then explore the equality $\cg(\oA)=\cg(\oB)$ for 
$\oA,\oB\in\CM(M,\Algebroidg)$. Let us choose an open cover 
$\Ucal=\{U_\alpha\}_\alpha$ for $M$ so that 
$\oA|_{U_\alpha}=\oX_\alpha\star 0$ and 
$\oB|_{U_\alpha}=\oY_\alpha\star 0$
for $\oX_\alpha,\oY_\alpha\in\Omega^0(U_\alpha,\Groupoidg)$.
It follows that, after passing to a refinement of the cover 
$\Ucal$, we can assume that  over $U_\alpha\cap U_\beta$ 
\[d_{\alpha}\circ(\oY_\alpha\circ\oY_\beta^{-1})=
(\oX_\alpha\circ\oX_\beta^{-1})\circ d_\beta
\ \ \ \ \text{for locally constant }d_\alpha\in 
\Omega^0(U_\alpha,\Gamma^\sg).\]
If we set 
$\oZ_\alpha=X_\alpha^{-1}\circ d_\alpha\circ Y_\alpha
\in\Omega^0(U_\alpha,\Groupoidg)$, it follows that 
$\oZ_\alpha=\oZ_\beta$ over $U_\alpha\cap U_\beta$. In particular,
$\oZ_\alpha$ is the restriction of a global section 
$\oZ\in\Omega^0(M,\Groupoidg)$ to $U_\alpha$. 
Note that $s_{\oZ}|_{U_\alpha}=s_{\oY_\alpha}=s_{\oA}|_{U_\alpha}$.
From $\oX_\alpha\circ \oZ|_{U_\alpha}=d_\alpha\circ \oY_\alpha$ 
it follows that
\[(\oZ\star\oA)|_{U_\alpha}=\oY_\alpha\star d_\alpha\star 0
=\oY_\alpha\star 0=\oB|_{U_\alpha}\ \ \ 
\Rightarrow\ \ \ \oZ\star \oA=\oB.\] 
If $\oZ\star\oA=\oB$ for some $\oZ\in\Omega^0(M,\Groupoidg)$,
it is also implied from the above argument that $\cg(\oA)=\cg(\oB)$. 
We thus obtain a well-defined injective map
\[\cg:\Mod(M,\Algebroidg,\Groupoidg)=\CM(M,\Algebroidg)/
\Omega^0(M,\Groupoidg)\ra \HC^1(M,\Gamma^\sg).\]

To complete the proof for germ cords, we then need to show that the 
map $\cg$ is surjective. Let us assume that a cocycle 
$\cg_{\alpha\beta}:U_\alpha\cap U_\beta\ra \Gamma^\sg$ 
represents an element of $\HC^1(M,\Gamma^\sg)$. We may further 
assume that $U_\alpha$ are all contractible and that 
$U_\alpha=\cup_{\gamma\neq \alpha}(U_\alpha\cap U_\gamma)$.
It is implied that there are source maps $s_\alpha:U_\alpha\ra \R$
such that 
\[s_{\cg_{\alpha\beta}}=s_\beta|_{U_\alpha\cap U_\beta}
\ \ \ \text{and}\ \ \ 
t_{\cg_{\alpha\beta}}=s_\alpha|_{U_\alpha\cap U_\beta}\ \ \ \ \ \ 
\forall\ \ \alpha,\beta.\] 
Choose a smooth partition of unity 
$\{\lambda_\alpha:U_\alpha\ra \R^{\geq 0}\}$
 subordinate to the cover $\Ucal=\{U_\alpha\}_\alpha$ of 
$M$ and define $\oY_\alpha\in\Omega^0(U_\alpha,\Groupoidg)$
by \[\oY_\alpha^{-1}(\h,x)
=\sum_{\gamma\neq \alpha} \lambda_\gamma(x)\cg_{\gamma\alpha}(\h,x)
\ \ \ \forall\ (\h,x)\in W_\alpha,\]
where $W_\alpha$ denotes an open neighborhood of 
\[\Delta_\alpha
=\{(s_\alpha(x),x)\in\R\times U_\alpha\ |\ x\in U_\alpha\}.\]
Note that the source map for the right-hand-side of the above 
equation stays equal to $s_\alpha$, and the expression on the 
right-hand-side is thus well-defined. Moreover, the derivative 
of $\oY_\alpha^{-1}$ with respect to $\h$ is positive and 
$\oY_\alpha^{-1}(x)\in \Groupoidg$ for all $x$.
In particular, we can define the inverse of this arrow, 
which would be $\oY_\alpha:U_\alpha\ra  \Groupoidg$. 
The target map for $\oY_\alpha$
is $t_{\oY_\alpha}=s_\alpha$, while its source map is 
\[s_{\oY_\alpha}(x)=t_{\oY_\alpha^{-1}}(x)
=\sum_{\gamma}\lambda_\gamma(x)\cg_{\gamma\alpha}(s_\alpha(x),x)=
\sum_{\gamma}\lambda_\gamma(x)s_\gamma(x).\]
This means that the source maps of $\oY_\alpha$ define a 
well-defined map $s:M\ra \R$ and that 
$s_{\oY_\alpha}=s|_{U_\alpha}$ for all $\alpha$. Let us set 
$\oA_\alpha=\oY_\alpha\star 0\in \CM(U_\alpha,\Algebroidg)$.
We then compute
\begin{align*}
&\oY_\beta^{-1}=\sum_\gamma \lambda_\gamma \cg_{\gamma\beta}
=\Big(\sum_\gamma \lambda_\gamma 
\cg_{\gamma\alpha}\Big)\circ \cg_{\beta\alpha}
=\oY_{\alpha}^{-1}\circ \cg_{\beta\gamma}\\
\Rightarrow\ \ \ 
&\oA_\alpha|_{U_\alpha\cap U_\beta}
=\oY_\alpha\star 0=(\cg_{\alpha\beta}\circ \oY_\alpha)\star 0
=\oY_\alpha\star(\cg_{\alpha\beta}\star 0)
=\oA_\beta|_{U_\alpha\cap U_\beta}.
\end{align*}
In particular, $\oA_\alpha\in\CM(U_\alpha,\Groupoidg)$ 
match over the intersections to give a global germ cord 
$\oA\in\CM(M,\Groupoidg)$. It is clear from the construction that
$\cg(\oA)$ is the cocycle we started with. This completes the proof 
for germ cords.\\

The proof for quantum cords is completely similar, as discussed 
below. If follows from the proof for the germ cords that one 
can associate a well-defined \v{C}ech cohomology class 
$\cq(\qA)\in\HC^1(M,\Gamma^\sq)$ to every 
$\qA\in\CM(M,\Algebroidq)$. If $\qA=\quot(\oA)$ then 
$\cq(\qA)=\quot(\cg(\oA))$. This gives a map
\[\cq:\Mod(M,\Algebroidq,\Groupoidq)=\CM(M,\Algebroidq)/
\Omega^0(M,\Groupoidq)\ra \HC^1(M,\Gamma^\sq).\]
We then need to show that $\cq$ is surjective. The key point is 
that given a cocycle 
\[\cq=\left[\{\cq_{\alpha\beta}:U_\alpha\cap U_\beta\ra 
\Gamma^\sq\}_{\alpha,\beta}\right]\in\HC^1(M,\Gamma^\sq)\]
we can construct the sections $\qY_\alpha:U_\alpha\ra \Groupoidq$
using a partition of unity so that $\{\qY_\alpha\star 0\}_\alpha$
match over the inetrsections, and give a global quantum cord 
$\qA\in\Omega^1(M,\Algebroidq)$. The equalities 
$\qA|_{U_\alpha}=\qY_\alpha\star 0$ imply that $\qA$ is locally 
trivial and hence an element of $\CM(M,\Algebroidq)$.
\end{proof}

\subsection{The classifying spaces}
Theorem~\ref{thm:classification} implies that the gauge equivalence 
classes of germ cords on $M$ are in correspondence with equivalence 
classes of Haefliger structures, with values in $\Gamma^\sg$. On the
other hand, the commutative diagram of 
Equation~\ref{eq:T-is-surjective} suggests that the space of 
equivalence classes of Haeflieger $\Gamma^\sq$-structures also has 
all rights to be studied as the genralization of space of 
foliations. In particular, the concordance classes of 
$\Gamma^\sg$ and $\Gamma^\sq$ 
structures on $M$ are in correspondence with the homotopy classes 
of maps from $M$ to the classifying spaces $B\Gamma^\sg$ and 
$B\Gamma^\sq$ associated with the groupoids $\Gamma^\sg$ and 
$\Gamma^\sq$, respectively.
\\

\begin{defn}\label{defn:concordance}
The germ cords $\oA_0,\oA_1\in \CM(M,\Algebroidg)$ are called 
{\emph{concordant}} if there is a {\emph{germ concord}} 
$\oA\in\CM(M\times[0,1],\Algebroidg)$ with 
$\oA|_{M\times\{i\}}=\oA_i$ for $i=0,1$. Similarly, 
$\qA_0,\qA_1\in \CM(M,\Algebroidq)$ 
are  {\emph{concordant}} if there is a {\emph{quantum concord}} 
$\qA\in\CM(M\times[0,1],\Algebroidq)$ with 
$\qA|_{M\times\{i\}}=\qA_i$ for $i=0,1$. We write 
$\oA_0\sim_\sg\oA_1$ if $\oA_0,\oA_1\in \CM(M,\Algebroidg)$ 
are concordant and write $\qA_0\sim_\sq\qA_1$ if 
$\qA_0,\qA_1\in \CM(M,\Algebroidq)$ are concordant. 
\end{defn}

If $\oA_0$ and $\oA_1$ are concordant, we can choose the (germ) 
concord $\oA$ connecting them so that $\imath_s^*\oA=\oA_0$
for $s\in[0,\epsilon)$ and $\imath_s^*\oA=\oA_1$
for $s\in(1-\epsilon,1]$. This is needed when we  glue 
the concords to show that concordance is an equivalence 
relation. \\

Our first observation addresses the compatibility of the concept of
concordance with the action of the gauge group on germ cords and 
quantum cords.

\begin{prop}\label{prop:gauge-equivalence-implies-homotopy}
If $\oA_0,\oA_1\in\CM(M,\Algebroidg)$ (or in 
$\CM(M,\Algebroidq)$) are gauge equivalent, then 
they are concordant. 
\end{prop}
\begin{proof}
Let us assume that $\oA_0,\oA_1\in\CM(M,\Algebroidg)$ and 
$\oA_1=\oY\star\oA_0$. For every point 
$x\in M$ the gauge function $\oY$ is given by $\oY(\h,y)\in\R$ for 
$y\in U_x\subset M$ and $\h\in \R$ 
for  a sufficiently small neighborhood $U_x$ of $x\in M$. 
We can then define
\[\oZ(\h,y,s)=\h+e^{\frac{s}{1-s}}(\oY(\h,y)-\h),\ \ \ \ 
\forall\ y\in U_x\subset  M,\ s\in [0,1],\ \h\in \R.\]
For every $(y,s)\in U_x\times [0,1]$, the above definition gives 
a function $\oZ(y,s)=\oc(\cdot,y,s)$ from $\R$ to $\R$, and 
it is not hard to show that $\oZ(y,s)$ is a 
diffeomorphism if $\oY(y)=\oY(\cdot,y)$ is a  diffeomorphism, e.g.
since its $\h$-derivative is positive. From this construction, 
we obtain a gauge element 
$\oZ\in\Omega^0(M\times[0,1],\Groupoidg)$. 
The restriction of $\oZ$ to $M\times\{0\}$ is the identity map, 
while the restriction of $\oZ$ to $M\times\{1\}$ is $\ob$. 
Let us abuse the notation and denote the pull-back of $\oA_0$
on $M\times[0,1]$ (using the projection map over the first factor) 
by $\oA_0$. Since $\oA_0\in\CM(M\times[0,1],\Algebroidg)$ it 
follows that 
\[\oA:=\oc\star\oA_0\in\CM(M\times [0,1],\Algebroidg),
\ \ \ \imath_0^*\oA=\oA_0\ \ \ \text{and}\ \ \ 
\imath_1^*\oA=\ob\star\oA_0=\oA_1.\]  
This completes the proof of the proposition for germ cords. 
The proof for quantum cords is completely similar.
\end{proof}

 The concordance classes of germ cords and
quantum cords form
\begin{align*}
\Ccal(M,\Algebroidg)=\CM(M,\Algebroidg)/\sim_\sg\ \ \ \ 
\text{and}\ \ \ \   
\Ccal(M,\Algebroidq)=\CM(M,\Algebroidq)/\sim,
\end{align*}
which are called  the {\emph{germ concordia}} and the 
{\emph{quantum concordia}} of $M$ respectively. If 
$f:M_1\ra M_2$ is a smooth map, we obtain the 
induced pull-back maps 
\begin{align*}
f^*:\Ccal(M_2,\Algebroidg)\ra 
\Ccal(M_1,\Algebroidg)\ \ \ \text{and}\ \ \ 
f^*:\Ccal(M_2,\Algebroidq)\ra 
\Ccal(M_1,\Algebroidq) 
\end{align*}
It follows from 
Proposition~\ref{prop:gauge-equivalence-implies-homotopy}
that the germ and quantum concordia are quotients of the 
moduli spaces $\Mod(M,\Algebroidg,\Groupoidg)$ and 
$\Mod(M,\Algebroidq,\Groupoidq)$. Correspondingly, there
are quotient maps
\begin{align*}
\pi_\sg:\Mod(M,\Algebroidg,\Groupoidg)\ra \Ccal(M,\Algebroidg)
\ \ \ \ \text{and}\ \ \ \ 
\pi_\sq:\Mod(M,\Algebroidq,\Groupoidq)\ra \Ccal(M,\Algebroidq).
\end{align*}

The converse of 
Proposition~\ref{prop:gauge-equivalence-implies-homotopy}
is true in some cases. The concept of a cord may be 
defined using the flat $1$-forms with values in $C^\infty(S^1)$,
which is the Lie algebra associated with $\Diff^+(S^1)$.
The corresponding cords are called the {\emph{circle cords}}.
The space of circle cords is denoted by $\CordobaS$. We take the 
following proposition as a justification for the relation between 
$\Mod(M,\Algebroidg,\Groupoidg)$ and $\Ccal(M,\Algebroidg)$.

\begin{prop}\label{prop:concordant-implies-gauge-equivalent}
If $\oA_0,\oA_1\in\CordobaS$ are concordant then they are 
gauge equivalent. In particular, there is an injective map
\[\rho_{S^1,M}:\ConcordiaS\lra 
\Hom(\pi_1(M),\Diff^+(S^1))/\Diff^+(S^1),\]
whose image is identified with the kernel of the Euler class 
obtstruction map
\[\Euler:\Hom(\pi_1(M),\Diff^+(S^1))/\Diff^+(S^1)\ra 
H^2(\pi_1(M),\Z).\] 
\end{prop}
\begin{proof}
Let $\oA_0,\oA_1\in \CordobaS$ be concordant 
and let $\oA\in\CM(M\times[0,1],C^\infty(\R))$ be a concord 
connecting $\oA_0$ to $\oA_1$. The concord $\oA$ is then 
given as 
\[\oA(\h,x,s)=\oA_s(\h,x)+\oB_s(\h,x)ds,\ \ \ 
\oA_s\in\CM(M,C^\infty(S^1)),\ \oB_s\in\Omega^0(M,C^\infty(S^1)).\]
In fact, if $\imath_s:M\ra M\times\{s\}\subset M\times[0,1]$ denotes
the inclusion map for each $s\in[0,1]$, the induced $1$-form 
$\oA_s=\imath_s^*\oA\in\Omega^1(M,C^\infty(S^1))$ is  a circle cord, 
since it is the pull-back of a circle cord. The equation 
$d\oA=\oA'\oA$ further implies that
\begin{align*}
&0=d\oA_s+\big(d\oB_s-\partial_s\oA_s\big)ds
+\oA_s\oA_s'+(\oA_s\oB_s'-\oB_s\oA_s')ds\ \ 
\Rightarrow\ \ \partial_s\oA_s=\nabla_{\oA_s}\oB_s,
\end{align*}
where $\partial_s$ denotes the differentiation with respect to $s$.
We can then use  the exponential map 
$\exp:C^\infty(S^1)\ra \Diff^+(S^1)$
to define $\oY\in\Omega^0(M\times [0,1],\Diff^+(S^1))$ so that 
$\oY_0(x)=\oY(\cdot,x,0)=Id_{S^1}$ and 
\[\frac{\partial \oY}{\partial s}(\h,x,s)=-\oB_s(\h,x)\ \ \ \ 
\forall\ (\h,x,s)\in S^1\times M\times [0,1].\]
The compactness of $S^1\times M\times [0,1]$ implies that 
$\oY$ is everywhere defined, and is uniquely determined by the 
above requirements. Let us define $\wt{\oA}_s=\oY_s\star \oA_0$.
In particular, $\wt\oA_0=\oA_0$. Fix $s\in[0,1]$ and 
set $\oZ_{\epsilon}=\oY_s^{-1}\circ\oY_{s+\epsilon}$.
It is then clear that 
$\wt\oA_{s+\epsilon}=\oZ_{\epsilon}\star \wt\oA_s$ 
and we can compute
\begin{align*}
\partial_s\wt\oA_s
&=\lim_{\epsilon\ra 0}
\frac{\wt\oA_s\circ\oZ_{s}-\oZ_{s}'\wt\oA_s-
d\oZ_{\epsilon}}{\epsilon\oZ_{\epsilon'}}\\
&=\wt\oA_s'\left(\partial_\epsilon\oZ_\epsilon|_{\epsilon=0}\right)
-d\left(\partial_\epsilon\oZ_\epsilon|_{\epsilon=0}\right)
-\wt\oA_s\left(\partial_\epsilon\oZ_\epsilon'|_{\epsilon=0}\right)\\
&=-\wt\oA_s'\oB_s+\wt\oA_s\oB_s'+d\oB_s
=\nabla_{\wt\oA_s}\oB_s.
\end{align*}
It follows from the two equations 
$\partial_s\oA_s=\nabla_{\oA_s}\oB_s$ and 
$\partial_s\wt\oA_s=\nabla_{\wt\oA_s}\oB_s$ and the initial time 
equality $\oA_0=\wt\oA_0$ that $\oA_s=\wt\oA_s$ for all 
$s\in[0,1]$. In particular, $\oA_1=\wt\oA_1=\oY_1\star \oA_0$ is 
gauge equivalent to $\oA_0$.\\

Now that the notions of concordance and gauge equivalence are
the same for circle cords, the last part of the proposition follows
from the proof of Theorem~\ref{thm:classification} and an standard 
observation that the kernel of the obstruction map $\Euler$ is 
identified with {\emph{horizontal}} foliations of $M\times S^1$.
\end{proof}

Since $B\Gamma^\sg$ naturally (and classically) arises from 
the study of Haefliger structures 
on a manifold $M$ up to concordance, some very interesting results 
are already available in the literature about the topology of 
$B\Gamma^\sg$. Mather and Thurston proved that $B\Gamma^\sg$ 
is $2$-connected \cite{Mather-on-Haefliger},
\cite{Thurston-classification}. Moreover, Thurston showed 
\cite{Thurston-GV} that $H_i(B\Gamma^\sg;\Z)=0$ for $i=0,1,2$ 
while there is  a surjection \[\GV:H_3(B\Gamma^\sg;\Z)\lra \R\]
given by the Godbillon-Vey invariant, c.f. \cite{Godbillon-Vey} and 
\cite{Bott-Haefliger}. This result is in complete contrast with the case of 
homogenous foliations~\cite{GV-for-homogeneous} and foliations 
admitting a projective transversal structure~\cite{Rigid-GV}. The 
reader is referred to \cite{Hurder-GV} and \cite{Ghys-GV-1,Ghys-GV-2}
for more on the construction of the Godbillon-Vey invariants.\\

It is also interesting to study the fiber of the projection map 
from $\Mod(M,\Algebroidg,\Groupoidg)$ over a point of 
$[M,B\Gamma^\sg]$. The structure of the intersection of 
this fiber with $\Fcal(M)$, especially in dimension $3$, 
is studied in a number of interesting papers. Near a taut 
foliation (of a closed $3$-manifold), the topology of this 
intersection is studied in \cite{Local-Topology}. 
Among more recent results, one can mention the work of 
Eynard-Bontemps~\cite{Eynard}, where she shows that  any 
two non-singular foliations on a $3$-manifold $M$ which correspond
to the same point of $[M,B\Gamma^\sg]$ are in the same 
connected component of $\Fcal(M)$, meaning that the 
corresponding integrable plane fields are in the same connected 
component among all integrable plane fields on $M$. It is 
still open whether the corresponding connected component 
is path connected or not.\\

Nevertheless, the relation between $\Gamma^\sq$ and $\Fol(M)$ is not 
studied in the literature, and the topology of $B\Gamma^\sq$ 
is not known. The homomorphism $\quot:\Gamma^\sg\ra \Gamma^\sq$ 
gives a continuous map $\quot:B\Gamma^\sg\ra B\Gamma^\sq$. 
Correspondingly, there is a composition map 
\[\quot:[M,B\Gamma^\sg]\ra [M,B\Gamma^\sq],\ \ \ \ \ 
\quot[f]=[\quot\circ f]\ \ \forall\ [f]\in[M,B\Gamma^\sg].\]
Let $\HC^1(M,\Gamma^\sg)/\sim_\sg$ denote the space 
of concordance classes of $\Gamma^\sg$-Haefliger structures on a 
manifold $M$. Similarly, let $\HC^1(M,\Gamma^\sq)/\sim_\sq$ 
denote the space of concordance classes of $\Gamma^\sq$-Haefliger 
structures on $M$. Under the identification of 
$\HC^1(M,\Gamma^\sg)/\sim_\sg$ with $[M,B\Gamma^\sg]$
and identification of $\HC^1(M,\Gamma^\sq)/\sim_\sq$ with 
$[M,B\Gamma^\sq]$, the map $\quot$ defined above is identified 
with the map induced by the map
$\quot$ which appears in the last column of the diagram in 
Equation~\ref{eq:T-is-surjective}.

\section{The cohomology theory of cords}
\label{sec:cohomology}
\subsection{The cohomology groups}
The algebroids $\Algebroidq$ and $\Algebroidq$ 
correspond to differential graded Lie algebras and attached cohomology 
theories which control the deformations of these algebroids, 
c.f.~\cite{Cohomology-Algebroid}.
Given a  cord $\oA\in\CM(M,\Algebroidg)$, let 
$\Omega^i_{s_\oA}(M,\Algebroidg)$ denote the subspace of 
$\Omega^i(M,\Algebroidg)$ which consists of section $\oE$ with 
$s_\oE=s_\oA$. For $\qA\in\CM(M,\Algebroidq)$ we may define 
$\Omega^i_{s_\qA}(M,\Algebroidq)$ in a similar way.
Define the  twisted differential 
\begin{align*}
&\nabla_\oA:\Omega^i_{s_\oA}(M,\Algebroidg)\lra 
\Omega^{i+1}(M,\Algebroidg)\ \ \ \ \ 
\nabla_\oA(\oB):=d\oB+[\oA,\oB].
\end{align*}
The new differential satisfies $\nabla_\oA\circ\nabla_{\oA}=0$ and 
may be used to define the cohomology groups 
$\HG^i(M,\oA)$. Similarly, we can define the cohomology
groups $\HQ^i(M,\qA)$ for $\qA\in\CM(M,\Algebroidq)$. We study 
the basic properties of these cohomology groups in this section.\\

The group of diffeomorphisms of $M$ acts on all objects considered 
above in a compatible way. Given an element $\phi:M\ra M$ in $\Diff^+(M)$ 
and $\oA,\oB\in\Omega^*(M,\Algebroidg)$, 
we have $\phi^*\oA, \phi^*\oB\in \Omega^*(M,\Algebroidg)$ and 
$[\phi^*\oA,\phi^*\oB]=\phi^*[\oA,\oB]$. 
Thus $\phi^*\CM(M,\Algebroidg)=\CM(M,\Algebroidg)$, i.e. 
the pull-back of a germ cord (respectively, a quantum cord) is 
another germ cord (respectively, another quantum cord). 
Moreover, it follows that
$\nabla_{\phi^*\oA}(\phi^*\oB)=\phi^*(\nabla_\oA(\oB))$ 
and we thus obtain the natural isomorphisms 
\begin{align*}
&\phi_{\Algebroidg}^*:\HG^k(M,\oA)\ra \HG^k(M,{\phi^*\oA})\ \ \ \ \ 
\text{and}\ \ \ \ \ 
\phi_{\Algebroidq}^*:\HQ^k(M,\qA)\ra \HQ^k(M,{\phi^*\qA}).
\end{align*}  

 Let us fix a pair of germ cords $\oA$ and 
 $\oB$ in $\CM(M,\Algebroidg)$  which are gauge equivalent. 
 There is a gauge element
 $\oY\in\Omega^0(M,\Groupoidg)$ such that $\oY\star \oA=\oB$.
 We then define the homomorphism
 \begin{align*}
& \Phi_{\oA\ra \oB}:\Omega^*_{s_{\oA}}(M,\Algebroidg)\ra 
\Omega^*_{s_{\oB}}(M,\Algebroidg),
\ \ \ \ \ \ \Phi_{\oA\ra \oB}(\od):=\frac{\od\circ \ob}{\ob'}.
 \end{align*}
Note that $\Phi_{\oA\ra \oB}$ is an isomorphism.
\begin{prop}\label{prop:cohomology}
For every two gauge equivalent germ cords 
$\oA,\oB\in \CM(M,\Algebroidg)$
 the following diagram is commutative.
\begin{displaymath}
\begin{diagram}
\Omega^*_{s_\oA}(M,\Algebroidg)&\rTo{\ \ \ \nabla_{\oA}\ \ \ }&
\Omega^{*+1}_{s_\oB}(M,\Algebroidg)\\
\dTo{\Phi_{\oA\ra \oB}}&
&\dTo_{\Phi_{\oA\ra \oB}}\\  
\Omega^*_{s_\oA}(M,\Algebroidg)&\rTo{\ \ \ \nabla_{\oB}\ \ \ }&
\Omega^{*+1}_{s_\oB}(M,\Algebroidg)
\end{diagram}.
\end{displaymath}
In particular, $\Phi_{\oA\ra \oB}$ defines an  isomorphism
$\Phi_{\oA\ra \oB}:\HG^*(M,\oA)\ra 
\HG^*(M,\oB)$.
\end{prop}
 \begin{proof}
This is a straight forward computation:
\begin{align*}
(\nabla_{\oB}\circ\Phi_{\oA\ra \oB})(\od)&=
d\Big(\frac{\od\circ\ob}{\ob'}\Big)
+\frac{\oA\circ\ob-d\ob}{\ob'}\Big(\od'\circ \ob
-\frac{(\od\circ\ob)\ob''}{(\ob')^2}\Big)\\
&\ \ \ \ -\Big(\oA'\circ\ob-\frac{d\ob'}{\ob'}
-\frac{(\oA\circ\ob-d\ob)\ob''}{(\ob')^2}\Big)
\frac{\od\circ\ob}{\ob'}\\
&=\Big(\frac{(d\od)\circ\ob+d\ob(\od'\circ\ob)}{\ob'}
-\frac{d\ob'(\od\circ\ob)}{(\ob')^2}\Big)\\
&\ \ \ \ +\frac{(\oA\circ\ob-d\ob)(\od'\circ\ob)
-(\oA'\circ\ob)(\od\circ\ob)}{\ob'}
+\frac{(d\ob')(\od\circ\ob)}{(\ob')^2}\\
&=\frac{(d\od+\oA\od'-\oA'\od)\circ\ob}{\ob'}\\
&=(\Phi_{\oA\ra\oB}\circ\nabla_\oA)(\od)
\end{align*}
 \end{proof}
 Proposition~\ref{prop:cohomology} implies that $\HG^*(M,\oA)$,
 for $\oA\in\CM(M,\Algebroidg)$ corresponding to a fixed foliation 
 $\Fol$, form a system of cohomology groups together with the 
 isomorphisms  $\Phi_{\oA\ra \oB}$, and it
 makes sense to talk about the natural cohomology group
\begin{align*}
\HG^*(M,\Fol)=\frac{\coprod_\oA \Hall^*(M,\oA)}{\sim},
\end{align*}
where the union is over all germ cords $\oA$ corresponding 
to the foliation $\Fol$ and 
 $x\sim y$ if $x\in \HG^*(M,\oA)$ for some $\oA$ and 
$y=\Phi_{\oA\ra \oB}(x)\in \HG^*(M,\oB)$
for some other germ cord $\oB$ corresponding 
to $\Fol$.\\

The cohomology groups are closed under Lie bracket. In fact, for given 
$\oc\in\Omega^k(M,\Algebroidg)$ and $\od\in\Omega^l(M,\Algebroidg)$ 
we have
\begin{align*}
\nabla_\oA[\oc,\od]
&=[d\oc,\od]+(-1)^k[\oc,d\od]+\oA\oc\od''-\oA\oc''\od
-\oA'\oc\od'+\oA'\oc'\od\\
&=[d\oc,\od]+(\oA\oc'\od'-\oA'\oc\od'-\oA\oc''\od+\oA''\oc\od)\\
&\ \ \ +(-1)^k[\oc,d\od]+(\oA\oc\od''-\oA''\oc\od
-\oA\oc'\od'+\oA'\oc'\od)\\
&=[\nabla_\oA(\oc),\od]+(-1)^k[\oc,\nabla_\oA(\od)].
\end{align*}
We thus obtain a well-defined bracket 
\begin{align*}
[\cdot,\cdot]:\HG^k(M,\oA)\otimes \HG^l(M,\oA)\ra 
\HG^{k+l}(M,\oA).
\end{align*}
Given a gauge function $\ob\in\Omega^0(M,\Groupoidg)$, let 
$\oB=\ob\star\oA$,  we thus have 
\begin{align*}
\Phi_{\oA\ra \oB}[\oc,\od]&=
[\Phi_{\oA\ra \oB}(\oC),\Phi_{\oA\ra \oB}(\oD)].
\end{align*}
In particular, we obtain well-defined graded Lie bracket maps
\begin{align*}
[\cdot,\cdot]:\HG^k(M,\Fol)\otimes \HG^l(M,\Fol)
\ra \HG^{k+l}(M,\Fol).
\end{align*}
The whole discussion may be repeated when we consider differential
forms with values in $\Algebroidq$ and the corresponding cohomology 
groups $\HQ^*(M,\qA)$. The outcome  is the differential graded Lie algebra
$\HQ^*(M,\Fol)=\oplus_k\HQ^k(M,\Fol)$.

\subsection{Cohomology groups of non-singular foliations}
Let us assume that $\oA=\oA_{a,V}\in\Omega^1(M,\gfrak)$
denote the germ cord with $s_\oA=0$ 
associated with a foliation $\Fol$ which is
constructed from a $1$-form $a$ and a transverse vector field 
$V$ so that $a(V)=-1$. Furthermore, let 
\[A=\quot(\oA)=e^{\h L}a\in\Omega^1(M,\algebroidq)\] 
denote the corresponding quantum cord, where
 $L=L_V$ denotes differentiation in the direction of $V$.
Let us denote the subset of $\Omega^*(M,\R)$ consisting of the forms 
$w$ with $\imath_V(w)=0$ by $\Omega^*_{a,V}(M,\R)$. It is then 
clear that $d_{a,V}$ restricts to a {\emph{Bott differential}}
\begin{align*}
d_{a,V}:\Omega^*_{a,V}(M,\R)\ra \Omega^{*+1}_{a,V}(M,\R).
\end{align*}
To see this, it is enough to note that if $\imath_V(w)=0$ then
\begin{align*}
\imath_V(d_{a,V}(w))&=\imath_V\left(dw+aL(w)-bw\right)
=L(w)+\imath_V(a)L(w)-a\imath_V(L(w))-\imath_V(b)w
=0.
\end{align*}
Here $b=L(a)$ is the derivative of $a$ in the direction of $V$.
We can then define $H^*_{a,V}(M,\R)$ as the cohomology of the 
chain complex $(\Omega^*_{a,V}(M,\R),d_{a,V})$:
\begin{equation}\label{eq:k-cohomology}
H^*_{a,V}(M):=
\frac{\left\{w\in\Omega^*(M,\R)\ \big|\ \imath_V(w)=0\ \ 
\text{and}\ \  d_{a,V}(w)=0\right\}}
{\left\{d_{a,V}(z)\ \big|\ z\in\Omega^{*-1}(M,\R)\ \ 
\text{and}\ \ \imath_V(z)=0\right\}}.
\end{equation}

\begin{thm}\label{thm:k-cohomology}
With the above notation fixed, we have
\begin{align*}
\HG^*(M,\Fol)\simeq\HQ^*(M,\Fol)\simeq H^*_{a,V}(M).
\end{align*}
\end{thm}
\begin{proof}
We first prove the isomorphism  $\HQ^*(M,\Fol)\simeq H^*_{a,V}(M)$.
Given $X=\sum_nx_n\h^n$ in $\Omega^*(M,\algebroidq)$  
we can inductively define 
\begin{align*}
&Y=\sum_ny_n\h^n\in\Omega^{*-1}(M,\algebroidq),
\ \ \ \ \ \ y_{n}=\begin{cases}
0 &\text{if}\ n=0,\\
\frac{1}{n}\imath_V(x_{n-1}+dy_{n-1})\ \ \ \ &\text{if}\ n\geq 1. 
\end{cases}
\end{align*}
From this definition, it follows that $\imath_V(Y)=\imath_V(Y')=0$ 
and $Y'=\imath_V(X+dY)$. This implies
\begin{align*}
\imath_V(X+\nabla_A(Y))=\imath_V\left(X+dY+AY'-A'Y\right)
=Y'+\imath_V(A)Y'-\imath_V(A')Y=0.
\end{align*}
In particular, if $\nabla_A(X)=0$ and  $X$ represents
a cohomology class in  $\HQ^*(M,\qA)$, after replacing $X$ with 
$X+\nabla_A(Y)$ which represents the same cohomology 
class as $X$, we can assume $\imath_V(X)=0$ 
(and thus $\imath_V(X')=0$). In particular,
$\imath_V(d(X))=L(X)$. Applying $\imath_V$ to the two sides of  
$\nabla_A(X)=0$ we find $L(X)-X'=0$, which is equivalent to 
$X=e^{\h L}x_0$. From $\nabla_A(X)=0$ we obtain
\begin{align*}
0&
=d\left(e^{\h L}x_0\right)+\left(e^{\h L}a\right)
\left(e^{\h L}L(x_0)\right)-\left(e^{\h L}b\right)
\left(e^{\h L}x_0\right)
=e^{\h L}\left(d(x_0)+aL(x_0)-bx_0\right)\\
\iff\ \ \ 0&=d(x_0)+aL(x_0)-bx_0.
\end{align*}
In particular, $x_0\in\Omega^{*}_{a,V}(M,\R)$ is in the kernel of 
$d_{a,V}$ and uniquely determines $X$, such that $\nabla_A(X)=0$ 
and $\imath_V(X)=0$.\\

Let us now assume that $X$ is of the form $\nabla_A(Y)$ and 
satisfies $\imath_V(X)=0$. We can then assume that $\imath_V(Y)=0$ 
as well, possibly after replacing $Y$ by some $Y+\nabla_A(Z)$. 
The above considerations imply that $X=e^{\h L}x_0$. If we look 
at the initial terms in the equation $X=\nabla_A(Y)$ we conclude
\begin{align*}
&x_0=d(y_0)+ay_1-by_0\ \ \  
\Rightarrow\ \ \ 0=\imath_V(x_0)
=L(y_0)+\imath_V(a)y_1-\imath_V(b)y_0
\ \ \ \Rightarrow\ \ \ y_1=L(y_0).
\end{align*}
Thus, $x_0=d_{a,V}(y_0)$, which completes the 
proof of the isomorphism $\HQ^*(M,\Fol)\simeq H^*_{a,V}(M)$.\\

The isomorphism $\HG^*(M,\Fol)\simeq H^*_{a,V}(M)$ is proved
in a completely similar manner, as sketched below.
Suppose that $\qA=\qA_{a,V}\in\CM(M,\algebroidq)$ correspond 
to the foliation $\Fol$.
If $\oa\in\Omega^*(M,\algebroidg)$ represents a cohomology class
in $\HG^*(M,\oA)$, we can construct $\ob$ so that
$\ob(0)=0$ and $\ob'=\imath_V(\oa+d\ob)$ is satisfied.
After replacing $\oa$ with $\oa+\nabla_\oA(\ob)$ we can assume 
that $\imath_V(\oa)=\imath_V(\oa')=0$. Let $x_0$ denote 
the initial term of $\oX$. We then have 
$d(x_0)+aL(x_0)-L(a)x_0=0$. In particular, $x_0\in\Ker(d_{a,V})$
and it uniquely determines $\oX$ in a neighborhood
of $0\in\R$ through the differential equation $\oX'=L(\oX)$, 
with the initial condition $\oX(0)=x_0$. Furthermore, for 
this solution we find 
that $\oE=\nabla_\oA(\oX)$ satisfies $\oE(0)=0$ and $\oE'=L(\oE)$.
This differential equation implies that $\oE$ vanishes in a 
neighborhood of $0\in\R$. Finally, if $\oa=\nabla_\oA(\ob)$,
we can assume that $\imath_V(\oX)=0$. After replacing
$\ob$ with $\ob+\nabla_\oA(\oc)$ if necessary, we can further 
assume that $\imath_V(\ob)=0$. Looking at the initial 
terms on the two sides of 
$\oa=\nabla_\oA(\ob)$ we find $y_1=L(y_0)$ and 
$x_0=d(y_0)+ay_1-by_0$, which means that $x_0=d_{a,V}(y_0)$. 
This observation completes the proof of the theorem.
\end{proof}

In the remainder of this subsection, we will focus on the 
computation of $\HQ^*(M,\Fol)$ in a number of special cases. 
The above theorem implies that the corresponding results remain 
valid for the groups $\HG^*(M,\Fol)$.

\begin{cor}\label{cor:top-is-zero}
For every transversely oriented codimension-one foliation $\Fol$ on 
the $n$-dimensional manifold $M$, $\HQ^n(M,\Fol)=0$.
\end{cor}
\begin{proof}
This is an immediate corollary of Theorem~\ref{thm:k-cohomology}.
\end{proof}

Let us now assume that $a\in\Omega^1(M,\R)$ is a closed nowhere 
zero one-from. Then $a$ defines a foliation $\Fol=\Fol^a$ on $M$ 
and the corresponding quantum cord is $A=a$. The foliation $\Fol$ 
may be lifted to the universal cover $\wt{M}$ of $M$ using the 
covering map $\pi:\wt{M}\ra M$ to give the foliation $\wt{\Fol}$. 
This foliation corresponds to the quantum cord $\wt{A}=\pi^* A$.
Let us denoted the leaf space of the foliation $\wt{\Fol}$ 
by $\wt{\Lcal}=\Lcal_{\wt{\Fol}}$. We can also
define $\Lcal=\Lcal_{\Fol}$ to be the quotient of $\wt\Lcal$ under
the covering map $\pi$. We call a function on $\Lcal$ 
{\emph{smooth}} if it lifts to a smooth function on $\wt\Lcal$. 
 In particular, the restriction of any such function to the 
 closure of any leaf $\ell$ of $\Fol$ is constant.
With the above notation fixed,
the group $\HQ^0(M,\Fol)$ may then be computed
in a relatively easy way, using Theorem~\ref{thm:k-cohomology}.\\
 
\begin{cor}\label{cor:zero-cohomology}
If $a\in\Omega^1(M,\R)$ is a closed nowhere zero one-from which 
gives the foliation $\Fol$, 
\[\HQ^0(M,\Fol)\simeq C^\infty(\Lcal_\Fol,\R).\]
\end{cor}

\begin{proof}
Theorem~\ref{thm:k-cohomology} identifies $\HQ^0(M,\Fol)$
with the kernel of $d_{a,V}$ (for a corresponding transverse vector 
field $V$). A function $f_0\in C^\infty(M,\R)$ is in the kernel of 
$d_{a,V}$ if and only if its restriction to every leaf of 
$\Fol$ is constant.
Such a function gives a section in $C^\infty(\Lcal_\Fol,\R)$.
Conversely, any function in $C^\infty(\Lcal_\Fol,\R)$ gives a smooth 
function from $M$ to $\R$ which remains constant on the leaves of 
$\Fol$, which is in the kernel of $d_{a,V}$.
\end{proof}

In fact, most of the above argument may be repeated for arbitrary 
foliations to compute their zero cohomology group. 
The equation $df_0+aL(f_0)-bf_0=0$ is satisfied in the 
transverse direction, i.e. the image of the left-hand-side under 
$\imath_V$ is automatically zero. The equation is thus equivalent 
to the equalities $df_0-f_0b=0$ on all leaves of $\Fol$. 
Note that the restriction of $b=L(a)$ to the leaves of 
$\Fol$ is closed, since $dL(a)=L(da)=L^2(a)a$. 
The $1$-form $b$ would then define the cohomology groups 
$H_{b}^i(\ell,\R)$ for every leaf $\ell$ of $\Fol$. 
For this purpose, we use the twisted differential 
\[d_{b}:\Omega^*(\ell,\R)\ra \Omega^{*+1}(\ell,\R),\ \ \ \ \  
d_{b}(X):=d(X)-L(b)X.\] 
The above argument shows that for every leaf $\ell$ of $\Fol$, 
$f_0$ is a section of $H_{b}^0(\ell,\R)$, which is zero unless 
$b$ is exact on $\ell$. If $b=dg_\ell$ on $\ell$, it follows that 
$f_0=c_\ell e^{g_\ell}$ for some constant $c_\ell\in\R$. 
In particular, if $g_0$ is not bounded above, the bounded function 
$f_0$ is forced to be zero.\\

The $1$-form $b=L(a)$ satisfies $da=ba$, and changing the vector 
field $V$ would correspond to choosing other $b$ with this property. 
If $db'=b'a$ for another $1$-form $b'$, it follows that $b'=b+ha$ 
for some function $h$. In particular, the restriction of $b$ to the 
leaves of $\Fol$ only depends on $a$. 
If $a$ is changed to $e^{h}a$, where $h$ is forced to be bounded 
above, the restriction of $b$ to the leaves of $\Fol$ is changed to 
$b+dh$. The set of points  $\Dcal=\Dcal_\Fol\subset \Lcal_\Fol$
 where the restriction of $b$ is  not of the form
$dg$ for some real valued function which is bounded above, is 
thus independent of the choice of $a$ and $b$, and only depends on 
the foliation $\Fol$ and functions in $H^0_{a,V}(M)$ vanish on 
$\Dcal$. Following this approach, every cohomology class 
$\qa\in \HQ^0(M,\Fol)$ may be studied using its restrictions to 
the leaves.

\section{Foliations of higher codimension}
\label{sec:general-case}
\subsection{The groupoids and the corresponding algebras}
\label{subsec:gauge-group-general}
In this section, we generalize our constructions in the previous 
sections to the case of foliations of higher codimension. The 
first step would be generalizing the Lie groupoids $\Groupoidg$ and 
$\Groupoidq$ and the corresponding algebroids
$\Algebroidg$ and $\Algebroidq$. Most computations remain completely
similar to the case of codimension one foliations.\\

Let us denote the groupoid of germs of local diffeomorphisms of 
$\R^k$ by $\Groupoidg_k$. The objects of $\Groupoidg_k$ are the 
points in $\R^k$ and the arrows from $x\in\R^k$ to $y\in\R^k$ are 
the germs of orientation preserving diffeomorphisms 
$f:\R^k\ra \R^k$ which send $x$ to $y$. The equivalence class of 
$f$ is denoted by $[f,x,y]_\sg$ or $[f,x]_\sg$. Note that 
$[f,x]_\sg=[g,x]_\sg$ if $f=g$ in an open 
neighborhood of  $x\in\R^k$. Similarly, we can define $\Groupoidq_k$
by requiring that $[f,x]_\sq=[g,x]_\sq$ if the Taylor expansions of 
$f$ and $g$ match at $x$. The equivalence class $[f,x]_\sq$ may then 
be represented by a formal power series
\[
[f,x]_\sq=\qY=\sum_{i=1}^k\sum_{\substack{I=(i_1,\ldots,i_k)
\in\Z^k\\ i_1,\ldots,i_k\geq 0}}
y_{i,I}(\h_1-x_1)^{i_1}\cdots (\h_k-x_k)^{i_k}\partial_i
=\sum_{i,I}y_{i,I}(\h-x)^I\partial_i.
\]
Here, $\partial_1,\ldots,\partial_k$ denote the unit vector of 
$\R^k$ and have a formal nature in the above expression. 
Moreover, we have 
\[y_{i,I}=\frac{\partial^{|I|}f_i}{(\partial \h_1)^{i_1}\cdots 
(\partial \h_k)^{i_k}}(x_1,\hdots,x_k),\ \ \ \ 
\text{where\ }f=(f_1,\ldots,f_k)\ \text{and}\ I=(i_1,\ldots,i_k).\]
Being a local diffeomorphism means
that the determinant of the matrix 
$\det(\qY)=\det(y_{i,j})_{i,j=1}^k$ is positive.
Correspondingly, we can define the algebroids $\Algebroidg_k$ 
and $\Algebroidq_k$ which are fibered over $\R^k$. The fiber of 
$\Algebroidq_k$ over $x\in\R^k$ consists of the formal power 
series $\qA=\sum_{i,I}a_{i,I}(\h-x)^I\partial_i$.
One can choose the discrete topology on the space of arrows from 
$x$ to $y$ in $\Groupoidg_k$ and $\Groupoidq_k$ to arrive at the 
groupoids $\Gamma^\sg_k$ and $\Gamma^\sq_k$. These two groupoids 
are the same as $\Groupoidg_k$ and $\Groupoidq_k$ (respectively) 
except that their topologies are different.
There are source maps and target maps
\begin{align*}
&s,t:\Groupoidg_k\ra \R^k, \ \ \ \  s,t:\Groupoidq_k\ra \R^n,\ \ \ \ 
s,t:\Algebroidg_k\ra \R^k \ \ \ \ \text{and}\ \ \ \ 
s,t:\Algebroidq_k\ra \R^n.
\end{align*}

The groupoids $\Groupoidg_k$ and $\Groupoidq_k$ act on 
$\Algebroidg_k$ and $\Algebroidq_k$, respectively, and the 
action is given  by 
\begin{align*}
\oY\star \oA:=(\oA\circ \oY)(\oY')^{-1}\ \ \ \ \ 
\forall\ \oY\in\Groupoidg_k,\ \oA\in\Algebroidg_k\ \ \ 
\text{such that}\ t_{\oY}=s_{\oA}.
\end{align*}
Here, $\oY'$ is the $k\times k$ matrix whose entries consist of 
different first order derivatives of $\oY=(\oY_1,\ldots,\oY_k)$ 
with respect to the variables $\h_1,\ldots,\h_k$. Since 
$\det(\oY)$ is positive, it follows that $\oY'$ is invertible
(both in the germ case and the quantum case).\\

\subsection{Germ cords and quantum cords}
There is a Lie bracket on $\Omega^*(M,\Algebroidg_k)$ 
(and an induced Lie bracket on $\Omega^*(M,\Algebroidq_k)$. 
For this purpose, we define 
\begin{align*}
\Big[\sum_{i=1}^k\oA_i\partial_i,\sum_{j=1}^k\oB_j\partial_j\Big]:=
\sum_{i,j}(\oA_i(\partial_i \oB_j)-\oB_i(\partial_i\oA_j))
\partial_j.
\end{align*}
The germ cords and quantum cords may then be defined as before. 
A germ cord is a smooth section $\oA\in\Omega^1(M,\Algebroidg_k)$ 
which satisfied $d\oA+\frac{1}{2}[\oA,\oA]=0$. This means that 
$\oA=(\oA_1,\ldots,\oA_k)$ and that 
$d\oA_i+\sum_j\oA_j(\partial_j\oA_i)=0$ for $i=1,\ldots,k$. The 
space of germ cords and quantum cords are denoted by 
$\CM(M,\Algebroidg_k)$ and $\CM(M,\Algebroidq_k)$, respectively. 
As before, $\Omega^0(M,\Groupoidg_k)$
and $\Omega^0(M,\Groupoidq_k)$ act on $\CM(M,\Algebroidg_k)$
and $\CM(M,\Algebroidq_k)$, respectively. 
Note that a quantum cord is always assumed to be the image of a 
germ cord. Alternatively, we always restrict our attention 
to locally trivial flat $1$-forms with values in $\Algebroidq_k$.
The quotients give the moduli spaces of germ cords and quantum cords
\begin{align*}
\Mod(M,\Algebroidg_k,\Groupoidg_k)
=\CM(M,\Algebroidg_k)/\Omega^0(M,\Groupoidg_k)
\ \ \ \ \text{and}\ \ \  
\Mod(M,\Algebroidq_k,\Groupoidq_k)
=\CM(M,\Algebroidq_k)/\Omega^0(M,\Groupoidq_k).
\end{align*}
 As in the commutative diagram of Equation~\ref{eq:T-is-surjective}, 
one can restrict attention to the cords with values in the fiber 
of $\Algebroidg_k$ over $0\in\R^k$ (or the fiber of $\Algebroidq_k$ 
over $0\in\R^k$). If $\oA=(\oA_1,\ldots,\oA_k)$ is such a 
germ cord, $\oA(0)$ is a $k\times k$ matrix with real values. 
If the determinant of this matrix is everywhere positive on $M$, 
$\oA$ corresponds to a smooth framed foliation $\Fol$ on $M$. 
Let us denote the space of all such framed foliations of 
codimension $k$  by $\Fcal_k(M)$.
The action of the gauge group $\Groupoidg_k$ preserves the 
foliation $\Fol$ associated with $\oA$ and we thus obtain the 
following commutative diagram. 
\begin{equation}\label{eq:T-is-surjective-general}
\begin{diagram}
\Fcal_k(M)&\rInto{\ \ \ \ \Ibb^\sg_k\ \ \ \ } &
\Mod(M,\Algebroidg_k,\Groupoidg_k)\\
\dTo{Id}&&\dTo{\quot}\\
\Fcal_k(M)&\rInto{\ \ \ \ \Ibb^\sq_k\ \ \ \ } &
\Mod(M,\Algebroidq_k,\Groupoidq_k)
\end{diagram}
\end{equation} 
 The proof of 
Theorem~\ref{thm:classification} may then be repeated to prove 
the following more general form of it. 
\begin{thm}\label{thm:classification-general}
There are natural classification maps 
\begin{align*}
\cg:\Mod(M,\Algebroidg_k,\Groupoidg_k)\ra 
\HC^1(M,\Gamma^\sg_k)\ \ \ \text{and}\ \ \ 
\cq:\Mod(M,\Algebroidq_k,\Groupoidq_k)\ra 
\HC^1(M,\Gamma^\sq_k)
\end{align*}
from the moduli spaces of germs cords and quantum cords to the 
space of $\Gamma^\sg_k$ and $\Gamma^\sq_k$ structures on a 
manifold $M$. This classification map induces the maps
\begin{align*}
\cg:\Ccal(M,\Algebroidg_k)
=\CM(M,\Algebroidg_k)/\sim_\sg\ra [M,B\Gamma^\sg_k]
\ \ \ \text{and}\ \ \ 
\cq:\Ccal(M,\Algebroidq_k)
=\CM(M,\Algebroidq_k)/\sim_\sq\ra [M,B\Gamma^\sq_k]
\end{align*}
from the germs concordia (concordance classes of germ cords) 
and quantum concordia (concordance classes of quantum cords) 
to the spaces of homotopy classes of maps from $M$ to 
the classifying spaces $B\Gamma^\sg_k$ and $B\Gamma^\sq_k$, 
respectively.
\end{thm}

 \bibliographystyle{siam} 

\end{document}